\documentclass[10pt]{amsart}

\usepackage{latexsym,amsfonts,amssymb,epsfig,verbatim}
\usepackage{amsmath, latexsym, graphics, textcomp,psfrag}
\usepackage{float}
\usepackage{amsthm}
\usepackage{lineno}

\usepackage[all,2cell,dvips]{xy}

\newcommand{\nc}{\newcommand}
\nc{\dmo}{\DeclareMathOperator}
\nc{\nt}{\newtheorem}

\nc{\ds}{\displaystyle}
\nc{\ens}{\ensuremath}

\theoremstyle{plain}
\nt{theorem}{Theorem}[section]
\nt{main}{Theorem}

\nt{lemma}[theorem]{Lemma}
\nt{cor}[theorem]{Corollary}
\nt{conj}[theorem]{Conjecture}
\nt{q}[theorem]{Question}
\nt{fact}[theorem]{Fact}
\nt{prop}[theorem]{Proposition}

\dmo{\Mod}{Mod}
\dmo{\PMod}{PMod}
\dmo{\Homeo}{Homeo}
\dmo{\Sym}{Sym}
\dmo{\Aut}{Aut}
\dmo{\Out}{Out}
\dmo{\Sp}{Sp}
\dmo{\SL}{SL}
\dmo{\GL}{GL}

\dmo{\cd}{cd}
\dmo{\vcd}{vcd}
\dmo{\colim}{colim}

\dmo{\Stab}{Stab}
\dmo{\support}{support}

\dmo{\hocolim}{hocolim}
\dmo{\glue}{glue}

\nc{\Z}{\mathbb{Z}}
\nc{\z}{\mathcal{Z}}
\nc{\C}{\mathcal{C}}
\nc{\V}{\mathcal{V}}
\nc{\R}{\mathbb{R}}
\nc{\Q}{\mathbb{Q}}
\nc{\M}{\mathcal{M}}
\nc{\N}{\mathcal{N}}
\nc{\I}{\mathcal{I}}
\nc{\J}{\mathcal{J}}
\nc{\K}{\mathcal{K}}
\nc{\hyp}{\mathbb{H}}
\nc{\T}{\mathcal{T}}
\nc{\Y}{\mathcal{Y}}
\nc{\A}{\mathcal{A}}
\nc{\F}{\mathcal{F}}
\renewcommand{\S}{\mathcal{S}}
\nc{\B}{\mathcal{B}}
\renewcommand{\P}{\mathcal{P}}

\nc{\Dmax}{D_{\mathrm{max}}}
\nc{\Dmin}{D_{\mathrm{min}}}

\nc{\p}[1]{\medskip\paragraph{{\bf #1}}}
\nc{\margin}[1]{\marginpar{\scriptsize #1}}

\nc{\bpf}{\begin{proof}}
\nc{\epf}{\end{proof}}

\nc{\bl}[1]{\begin{list}{{\bf #1}}{
\setlength{\leftmargin}{.5in}
\setlength{\rightmargin}{.5in}
\setlength{\itemsep}{0ex plus 0ex minus 0ex}
}
}

\nc{\el}{\end{list}}

\include{diagram}

\begin{document}


\input{epsf.sty}


\title{The dimension of the Torelli group}

\author{Mladen Bestvina}
\author{Kai-Uwe Bux}
\author{Dan Margalit}

\address{Mladen Bestvina: Department of Mathematics\\ University of Utah\\ 155 S 1400 East \\ Salt Lake City, UT 84112-0090}
\email{bestvina@math.utah.edu}

\address{Kai-Uwe Bux: Department of Mathematics\\ University of Virginia\\ Kerchof
  Hall 229\\ Charlottesville, VA 22903-4137}
\email{kb2ue@virginia.edu}

\address{Dan Margalit: Department of Mathematics\\ University of Utah\\ 155 S 1400 East \\ Salt Lake City, UT 84112-0090}
\email{margalit@math.utah.edu}

\thanks{The first and third authors gratefully acknowledge support by
  the National Science Foundation.}



\maketitle

\begin{center}\today\end{center}

\begin{abstract}
We prove that the cohomological dimension of the Torelli group for a
closed connected orientable surface of genus $g \geq 2$ is equal to
$3g-5$.  This answers a question of Mess, who proved the lower
bound and settled the case of $g=2$.  We also find the
cohomological dimension of the Johnson kernel (the subgroup of the
Torelli group generated by Dehn twists about separating curves) to be
$2g-3$.  For $g \geq 2$, we prove that the top dimensional
homology of the Torelli group is infinitely generated.  Finally, we
give a new proof of the theorem of Mess that gives a precise
description of the Torelli group in genus 2.   The main tool
is a new contractible complex, called the ``complex of cycles'', on which the Torelli group acts.
\end{abstract}

\section{Introduction}

Let $S=S_g$ be a closed surface of genus $g$ (unless specified
otherwise, we take all surfaces to be connected and orientable).  Let
$\Mod(S)$ be the \emph{mapping class group} of $S$, defined as
$\pi_0(\Homeo^+(S))$, where $\Homeo^+(S)$ is the group of orientation
preserving homeomorphisms of $S$.  The \emph{Torelli group} $\I(S)$ is
the kernel of the natural action of $\Mod(S)$ on $H_1(S,\Z)$.  This
action is symplectic---it preserves the algebraic intersection
number---and it is a classical fact that $\Mod(S)$ surjects onto
$\Sp(2g,\Z)$.  All of this information is encoded in the following
short exact sequence:
\[ 1 \to \I(S) \to \Mod(S) \to \Sp(2g,\Z) \to 1. \]

\p{Cohomological dimension.} For a group $G$, we denote by $\cd(G)$
its \emph{cohomological dimension}, which is the supremum over all $n$
so that there exists a $G$-module $M$ with $H^n(G,M) \neq 0$.

\begin{main}
\label{main:torelli}
For $g \geq 2$, we have $\cd(\I(S_g)) = 3g-5$.
\end{main}

Since $\Mod(S_0)=1$, we have $\I(S_0)=1$.  Also, it is a classical
fact that $\Mod(S_1) \cong \Sp(2,\Z) = \SL(2,\Z)$, and so $\I(S_1)$ is
also trivial.

The lower bound for Theorem~\ref{main:torelli} was already given by Mess in
1990 \cite{mess}.

\begin{theorem}[Mess]
\label{thm:lower}
For $g \geq 2$, we have $\cd(\I(S_g)) \geq
3g-5$.
\end{theorem}

This theorem is proven by constructing a subgroup of $\I(S_g)$ which
is a Poincar\'e duality subgroup of dimension $3g-5$.  We recall
Mess's proof in Section~\ref{section:mess}.  Mess explicitly asked if
the statement of Theorem~\ref{main:torelli} is true
\cite{mess}.

We also study the \emph{Johnson kernel}, denoted $\K(S_g)$, which is the subgroup of $\I(S_g)$ generated by Dehn twists about
separating curves.  Powell showed that $\K(S_2)$ is equal to
$\I(S_2)$ \cite{jp}.  On the other hand, Johnson proved that the index
of $\K(S_g)$ in $\I(S_g)$ is infinite when $g \geq 3$
\cite{djabelian}, answering a question of Birman.

The following theorem answers a question of Farb \cite[Problem 5.9]{farb}.

\begin{main}
\label{main:kg}
For $g \geq 2$, we have $\cd(\K(S_g)) = 2g-3$.
\end{main}

$\K(S_g)$ has a free abelian subgroup of rank $2g-3$ when $g \geq 2$
(see Section~\ref{section:mess}).  From the inclusions $\Z^{2g-3} <
\K(S_g) < \I(S_g)$, and Theorem~\ref{main:torelli}, it
immediately follows that $2g-3 \leq \cd(\K(S_g)) \leq 3g-5$ (see
Fact~\ref{fact:cd subgp} below).

We can think of Theorems~\ref{main:torelli} and~\ref{main:kg} as
giving the smallest possible dimensions of Eilenberg--MacLane spaces
for $\I(S)$ and $\K(S)$---the so-called \emph{geometric dimension}.
Indeed, it is a theorem of Eilenberg--Ganea, Stallings, and Swan that
if $\cd(G) \neq 2$ for some group $G$, then $\cd(G)$ is the same as
the geometric dimension \cite{eg,jrs,rgs}.

For a group with torsion, such as $\Mod(S)$, the cohomological
dimension is infinite.  However, if a group $G$ has a torsion free
subgroup $H$ of finite index, then we can define $\vcd(G)$, the \emph{virtual
cohomological dimension} of $G$, to be $\cd(H)$.  It is a theorem of Serre
that $\vcd(G)$ does not depend on the choice of $H$ \cite[Th\'eor\`eme 1]{jps}.  In 1986, Harer 
proved that $\vcd(\Mod(S_g)) = 4g-5$ \cite{jlh}, so we see that there
is a gap of $g$ between $\cd(\I(S_g))$ and $\vcd(\Mod(S_g))$.

Let $\I^k(S_g)$ be the subgroup of $\Mod(S_g)$ consisting of elements
that act trivially on a fixed $2k$-dimensional symplectic subspace of
$H_1(S,\Z)$.

\begin{conj}
\label{conj:ik}
For $g \geq 2$ and $0 \leq k \leq g$, $\vcd(\I^k(S_g)) = 4g-5-k$.
\end{conj}

We discuss this conjecture further in Section~\ref{section:mess}.

The groups $\Mod(S_g)$, $\I(S_g)$, and $\K(S_g)$ are the first three groups in a
series of groups associated to $S_g$.  Let $\pi=\pi_1(S_g)$, and let
$\pi^k$ be the $k^{\mbox{\tiny th}}$ term of the lower central series
of $\pi$, that is, $\pi^0=\pi$ and $\pi^k=[\pi,\pi^{k-1}]$.  The
\emph{Johnson filtration} of $\Mod(S)$ is the sequence of groups
$\{\N_k(S_g)\}$ given by
\[ \N_k(S_g) = \ker \{ \Mod(S) \to \Out(\pi/\pi^k) \}. \]
It follows from definitions that $\N_0(S_g)=\Mod(S_g)$ and
$\N_1(S_g)=\I(S_g)$, and it is a deep theorem of Johnson that
$\N_2(S_g)=\K(S_g)$ \cite{dj2}.

For $g \geq 2$ and any $k$, Farb gives the lower bound $g-1$ for $\cd(\N_k(S_g))$ by constructing
a free abelian subgroup of $\N_k(S_g)$ of rank $g-1$ \cite{farb}.  From
this, and Theorem~\ref{main:kg}, it follows that, for $g \geq
2$ and any $k \geq 2$, we have $g-1 \leq \cd(\N_k(S_g)) \leq 2g-3$
(again apply Fact~\ref{fact:cd subgp}).  The following question was already asked by Farb \cite[Problem 5.9]{farb}.

\begin{q}
What is $\cd(\N_k(S_g))$ for $g \geq 3$ and $k \geq 3$?
\end{q}

\p{Infinite generation of top homology.} Besides the question of
cohomological dimension, one would also like to know in which
dimensions the homology is (in)finitely generated.  For instance, if
the second homology of a group is infinitely generated, then that
group is not finitely presented.  As of this writing, it is a central
open problem to determine whether or not $\I(S_g)$ is finitely
presented or not.  We have the following theorem.

\begin{main}
\label{main:top}
For $g \geq 2$, the group $H_{3g-5}(\I(S_g),\Z)$ is infinitely
generated.
\end{main}

In Kirby's problem list, Problem 2.9(B) (attributed to Mess) is to
find the largest $k=k_g$ so that $\I(S_g)$ has an Eilenberg--MacLane
space with finite $k$-skeleton \cite{rk}.  An immediate consequence of
Theorem~\ref{main:top} is that $k_g\leq3g-6$.  It follows from a
theorem of Mess (see below) that $k_3 \leq 2$.

Theorem~\ref{main:top} also solves a problem of Farb \cite[Problem 5.14]{farb}.

We do not know if the top dimensional homology of $\K(S_g)$ is finitely generated
for $g \geq 3$.

\p{Genus 2.} It is a theorem of Stallings and Swan that a group of
cohomological dimension 1 is free \cite{jrs,rgs}.  Thus,
Theorem~\ref{main:torelli} implies that $\I(S_2)=\K(S_2)$ is a free
group.  A celebrated theorem of Mess, Theorem~\ref{main:mess} below,
gives a much more precise picture of this group.

For the statement, note that each separating simple closed curve in
$S_2$ gives rise to an algebraic splitting of $H_1(S_2,\Z)$: the two
subspaces are the ones spanned by the curves on the two different
sides of the separating curve.  Each component of the splitting is a
2-dimensional symplectic subspace of $H_1(S_2,\Z)$, and so we call
such a splitting a \emph{symplectic splitting} of $H_1(S_2,\Z)$.  Note
that two separating curves give rise to the same symplectic splitting
if they differ by an element of $\I(S_2)$.

\begin{main}[Mess]
\label{main:mess}
$\I(S_2)$ is an infinitely generated free group, with one Dehn twist
generator for each symplectic splitting of $H_1(S_2,\Z)$.
\end{main}

We give a new proof of this theorem in Section~\ref{section:genus 2}.
While Mess's original proof is rooted in algebraic geometry, our proof
is confined entirely to the realm of geometric group theory.

\p{Strategy.} All of our theorems cited above are proven by studying
the actions of $\I(S_g)$ and $\K(S_g)$ on a new complex $\B(S_g)$,
called the ``complex of cycles'', which we construct in
Section~\ref{section:complex}.  Vertices of the complex correspond to
integral representatives in $S_g$ of some fixed element of
$H_1(S_g,\Z)$.  The key feature is the following fact, proven in
Sections~\ref{section:borrowing complex} and~\ref{section:chambers}.

\begin{main}
\label{main:b}
The complex $\B(S_g)$ is contractible for $g \geq 2$.
\end{main}

The basic idea is to use a nerve-cover argument to show that $\B(S_g)$
is homotopy equivalent to Teichm\"uller space.

In light of Theorem~\ref{thm:lower} (and the discussion after
Theorem~\ref{main:kg}), Theorems~\ref{main:torelli} and~\ref{main:kg}
reduce to the following:
\[ \cd(\I(S_g)) \leq 3g-5 \quad \textrm{ and } \quad \cd(\K(S_g)) \leq
2g-3. \]
For a group $G$ acting on a contractible cell complex $X$, Quillen proved that
\[ \cd(G) \leq \sup \{ \cd(\Stab_{G}(\sigma)) + \dim(\sigma) \}\]
where $\sigma$ ranges over the cells of $X$ \cite[Proposition
11]{jps}.  We take $X$ to be $\B(S_g)$, and we take $G$ to be either
$\I(S_g)$ or $\K(S_g)$.  We then show in Section~\ref{section:stabilizers}
that the right hand side is bounded above by $3g-5$ for $\I(S_g)$ and
by $2g-3$ for $\K(S_g)$.

In genus 2, the picture is very explicit---we give a picture of
$\B(S_2)/\I(S_2)$ in Figure~\ref{g2fareypic}.  Thus, we are able to
deduce Theorem~\ref{main:mess} by a careful analysis of this quotient.
For a free group, infinite generation of the group is the same as
infinite generation of its first homology.  As such,
Theorem~\ref{main:mess} gives the genus 2 case of
Theorem~\ref{main:top}.

For the proof of Theorem~\ref{main:top} we again appeal to the action
of $\I(S_g)$ on $\B(S_g)$.  It is a general fact that if a group $G$
acts on a contractible cell complex $X$ and satisfies the condition
\[ \sup \{ \cd(\Stab_{G}(\sigma)) + \dim(\sigma) \} = d \]
as above, then the $d^{\mbox{\tiny th}}$ integral homology of the
stabilizer in $G$ of a vertex of $X$ injects into $H_d(G,\Z)$ (see
Fact~\ref{fact:cartan} in Section~\ref{section:spectral}).  In
Section~\ref{section:top}, we use
an inductive argument---with base case Theorem~\ref{main:mess}---to
prove that the $(3g-5)^{\mbox{\tiny th}}$ homology of the stabilizer
in $\I(S_g)$ of a particular kind of vertex of $\B(S_g)$ is infinitely
generated.

\p{History.} The classical fact that $\I(S_1)$ is trivial was proven by
Dehn in the 1920's \cite{md}.  In 1983, Johnson proved that $\I(S_g)$
is finitely generated for $g \geq 3$ \cite{dj1}, and in 1986,
McCullough--Miller proved that $\I(S_2)$ is not finitely generated
\cite{mcm}.  Mess improved on this in 1992 by proving
Theorem~\ref{main:mess} \cite{gm}.  At the same time, Mess
showed that $H_3(\I(S_3),\Z)$ is not finitely generated (Mess credits
the argument to Johnson--Millson) \cite{gm}.  In 2001, Akita proved
that $H_\star(\I(S_g),\Z)$ is not finitely generated for $g \geq 7$
\cite{ta}.

Mess proved Theorem~\ref{thm:lower}, the lower bound for
$\cd(\I(S_g))$, in 1990 \cite{mess}.  The authors recently proved that
$\cd(\I(S_g)) < 4g-5 = \vcd(\Mod(S_g))$ \cite{bbm} (Harer's result
already implied that $\cd(\I(S_g)) \leq 4g-5$).

In the case of the Johnson kernel, less is known.  As
$\K(S_2)=\I(S_2)$, we have the theorem of Mess in this case.
Biss--Farb proved in 2006 that $\K(S_g)$ is not finitely generated for
$g \geq 3$ \cite{bf}.  An open question of Morita asks
whether or not $H_1(\K(S_g))$ is finitely generated for $g \geq 3$.

In an earlier paper, the authors studied the Torelli subgroup of
$\Out(F_n)$, that is, the group $\T_n$ of outer automorphisms of a
free group of rank $n$ that act trivially on the first homology of the
free group \cite{bbm}.  The main results of that paper---that $\cd(\T_n)=2n-4$ and that $H_{2n-4}(\T_n,\Z)$ is infinitely generated---are
analogous to Theorems~\ref{main:torelli} and~\ref{main:top} of the
present paper.  While proofs for corresponding theorems about
$\GL(n,\Z)$, $\Out(F_n)$, and $\Mod(S)$ often run parallel, at least on
a philosophical level, we encounter here a situation where this
paradigm does not hold.  For instance, the construction of the complex
of cycles in this paper is predicated on the fact that there can
be more than one shortest representative for a first homology class in
a hyperbolic surface; in a metric graph there is a exactly one
shortest representative for a first homology class.  Also, the proof
of the upper bound for $\cd(\T_n)$ does not (in any obvious way) give
the correct upper bound for $\cd(\I(S))$.

\p{Acknowledgements.} We would like to thank Ken Bromberg for many
discussions on Teichm\"uller theory.  We are also grateful to Tara
Brendle, Benson Farb, Chris Leininger, Justin Malestein, and Andy
Putman for helpful comments and conversations.


\section{The complex of cycles}
\label{section:complex}

After introducing some definitions and terminology, we will define the
complex of cycles for $S=S_g$.

\p{Cycles and curves.}   A 1-cycle in $S$ is a finite
formal sum
\[ \sum k_ic_i \]
where $k_i \in \R$, and each $c_i$ is an oriented simple closed curve in $S$;
the set $\{c_i: k_i \neq 0\}$ is called the \emph{support}.  We say
that the 1-cycle is \emph{simple} if the curves of the support are
pairwise disjoint, and we say that it is \emph{positive} if each $k_i$
is positive.  If we want to emphasize that the $k_i$ are in $\R$,
$\Q$, or $\Z$, we can call a cycle \emph{real}, \emph{rational}, or
\emph{integral}.

A \emph{basic cycle} is a simple positive 1-cycle $\sum k_i c_i$ with
the property that its support forms a linearly independent subset of
$H_1(S,\R)$.  If $x$ is an element of $H_1(S,\Z)$, a basic cycle for $x$ (that
is, a basic cycle representing $x$) must be integral.

Let $\S$ be the set of isotopy classes of oriented curves in $S$.  By
taking isotopy classes, there is a natural map from the set of
1-cycles in $S$ to $\R^\S$.  Formally, a point of $\R^\S$ is a
function from $\S$ to $\R$, and the image of the 1-cycle $\sum k_ic_i$
in $\R^\S$ is the function given by
\[
s \mapsto \displaystyle\sum_{c_i \in s} k_i.
\]

A \emph{multicurve} in the surface $S$ is a nonempty collection of
disjoint simple closed curves in $S$ that are isotopically nontrivial and
isotopically distinct.

\p{The construction.} The first step is to fix an arbitrary nontrivial
element $x$ of $H_1(S,\Z)$.  Next, let $\M$ be the set of isotopy
classes of oriented multicurves $M$ in $S$ with the property that each
oriented curve of $M$ appears in the support of some basic cycle for
$x$ supported in $M$.  The set $\M$ has a natural partial ordering
under inclusion (the orientation is important here).  As such, we can
think of $\M$ as a category where the morphisms are inclusions.

Given $M \in \M$, let $P_M$ be the polytope in $\R^\S$ given by the
convex hull of the images of the basic cycles for $x$ supported (with
orientation) in $M$.  The collection $\{P_M\}$ is a category under
inclusion, which we denote $\F$.  The functor given by
\[ M \mapsto P_M \]
is an isomorphism from $\M$ to $\F$.  We define the \emph{complex of cycles} $\B(S)$ as
the geometric realization of the colimit of this functor:
\[ \B(S) = \left |\mathop{\colim}_{M \in \M}\ \{ P_M\} \right |\]
(the choice of $x$ is suppressed in the notation).  We can thus regard
$\F$ as the category of cells of $\B(S)$.

Informally, $\B(S)$ is obtained from the disjoint union of the
elements of $\F$ by identifying faces that are equal in $\R^\S$.

Since $\I(S)$ acts trivially on $H_1(S,\Z)$, it acts on $\B(S)$.  In
the next two sections, we show that $\B(S)$ is homotopy equivalent to
Teichm\"uller space, and so it is contractible.

\p{Examples of cells.}  Consider the picture on the left hand side of
Figure~\ref{cellspic}.  Say that the homology classes of the curves
$a$, $b$, and $c$, are $x$, $y$, and $x-y$, respectively.  There are
exactly two basic cycles for $x$ supported on this multicurve, namely $a$
and $b+c$.  Thus, the associated cell of $\B(S_2)$ is an edge.

The three curves in the picture in the middle of Figure~\ref{cellspic}
are all homologous.  Thus, we may assume that they are all in the
homology class $x$.  It follows that the basic cycles are simply $a$,
$b$, and $c$, and the corresponding cell of $\B(S_4)$ is a 2-simplex.

Finally, we can think of the situation in the right hand side of the
figure as two independent copies of the first example.  Therefore, we
get the product of an edge with itself, or, a square.

\begin{figure}[htb]
\psfrag{a}{$a$}
\psfrag{b}{$b$}
\psfrag{b+c}{$b+c$}
\psfrag{d}{$d$}
\psfrag{e}{$e$}
\psfrag{f}{$f$}
\psfrag{c'}{$c'$}
\psfrag{c}{$c$}
\psfrag{a+b+d+e}{$a+b+d+e$}
\psfrag{c+d+e}{$c+d+e$}
\psfrag{c+f}{$c+f$}
\psfrag{a+b+f}{$a+b+f$}
\centerline{\includegraphics[scale=.5]{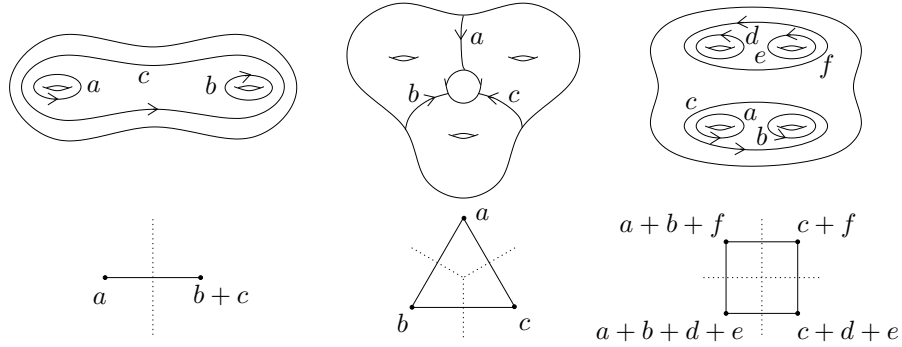}}
\caption{Multicurves with their associated cells.  The
  dotted lines indicate the (local) subdivision of $\T(S)$ into
  chambers (see Section~\ref{section:contractible}).}
\label{cellspic}
\end{figure}

While the examples of Figure~\ref{cellspic} are, in a sense, canonical
examples, the cells of $\B(S)$ come in a
wide variety of shapes.  Let $a$, $b$, $c$, $d$, $e$, and $f$ be the
oriented curves in $S_6$ shown in the left hand side of 
Figure~\ref{pentagonpic}, and consider, for instance, the homology
class $x=[d]+2[e]+[f]$.  There are relations $[a]+[b]+[c]=[e]$ and
$[a]+[b]+[d]=[f]$.  One can check that the resulting cell is
a pentagon, as shown in the right hand side of
Figure~\ref{pentagonpic}.

\begin{figure}[htb]
\psfrag{a}{$a$}
\psfrag{b}{$b$}
\psfrag{c}{$c$}
\psfrag{d}{$d$}
\psfrag{e}{$e$}
\psfrag{f}{$f$}
\psfrag{(0,0)}{$3a+3b+2c+2d$}
\psfrag{(2,0)}{$a+b+2d+2e$}
\psfrag{(2,1)}{$d+2e+f$}
\psfrag{(1,2)}{$c+e+2f$}
\psfrag{(0,2)}{$a+b+2c+2f$}
\centerline{\includegraphics[scale=.5]{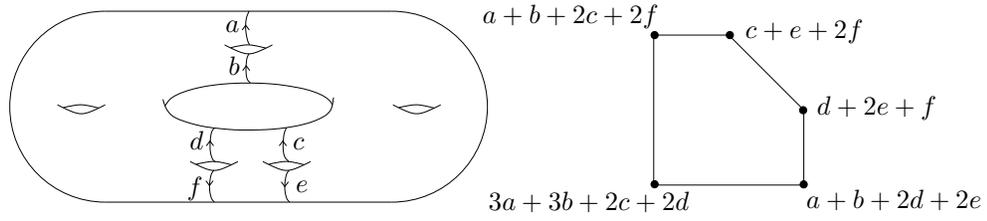}}
\caption{Curves that give rise to a pentagonal cell of
  $\B(S_6)$.  The picture is drawn in the ``$ef$-plane''.}
\label{pentagonpic}
\end{figure}

\p{Dimension.} We now characterize the dimension of a cell of $\B(S)$, and compute the dimension of the complex.

Given $M \in \M$, let $D=D(M)$ be the dimension of the span of $M$ in
$H_1(S,\R)$, let $N$ be the number of components of $S-M$, and let
$B=B(M)$ be the dimension of the associated cell of $\B(S)$.

\begin{lemma}
\label{lemma:cell dimension}
We have $B=|M|-D=N-1$.
\end{lemma}

\begin{proof}

There is a natural map from the real vector space spanned by the
curves of $M$ to $H_1(S,\Z)$.  The number $B$ is exactly the dimension
of this kernel, which is the same as $|M|-D$.  The second equality
follows from basic algebraic topology.
\end{proof}

The following proposition is not used anywhere in the paper, but it is a
basic fact about the complex of cycles.

\begin{prop}
For $g \geq 2$, the dimension of $\B(S_g)$ is $2g-3$.
\end{prop}

\begin{proof}

By Lemma~\ref{lemma:cell dimension}, the dimension of the cell of
$\B(S_g)$ associated to the oriented multicurve $M$ is one less than
the number of connected components of $S-M$.  Since each component of
$S-M$ has negative Euler characteristic, the largest possible number
of components is $-\chi(S)=2g-2$; this is when $M$ is a pants
decomposition.  Thus, the dimension of a cell of $\B(S_g)$ is at most
$2g-3$.  On the other hand, by choosing a pants decomposition of $S$ that supports $x$ and consists only of nonseparating curves, we obtain a cell of dimension $2g-3$.
\end{proof}


\section{The complex of cycles via minimizing cycles}
\label{section:borrowing complex}

The goal of this section is to give an alternate characterization of
$\B(S)$ that is related to Teichm\"uller space.  The idea is to
decompose Teichm\"uller space into regions indexed by the shortest
representatives of the homology class $x$ chosen in
Section~\ref{section:complex}.

\p{Teichm\"uller space.} For a surface $S$ with $\chi(S) < 0$ and no
boundary, we think of Teichm\"uller space $\T(S)$ as the space of
equivalence classes of marked hyperbolic surfaces $\{(X,f)\}$, where
$X$ is a complete hyperbolic surface, $f:S\to X$ is a homeomorphism,
and $(X,f) \sim (Y,g)$ if $f \circ g^{-1}$ is isotopic to an isometry;
see \cite{it} for an introduction.

There is a natural action of $\Mod(S)$ on $\T(S)$ given by the formula
$g \cdot (X,f) = (X,f \circ g^{-1})$.

It is a classical theorem that $\T(S)$ is diffeomorphic to
$\R^{-3\chi(S)}$ when $\chi(S) < 0$.  In particular, and most
important for us, $\T(S)$ is contractible.

\p{Lengths of curves.} Given a 1-cycle $c = \sum
k_ic_i$ in $S$, and a point $X=(X,f)$ in $\T(S)$, we can define the
\emph{length of $c$ in $X$} as
\[ \ell_X(c) = \sum |k_i|\ell_X(c_i) \]
where $\ell_X(c_i)$ is the length of the geodesic representative of
$f(c_i)$ in the hyperbolic metric of $X$.  It
is an important and basic fact that the function $\ell_\cdot(c) :
\T(S) \to \R_+$ is continuous.


\subsection{Minimizing cycles}
\label{section:min cycles}

Let $S=S_g$.  As in the definition of $\B(S)$, there is a fixed element $x$ of $H_1(S,\Z)$, chosen once and for all.  Given $X \in
\T(S)$, we denote by $|x|_\R$ the infimum of lengths in $X$ of all
real cycles in $S$ representing $x$, and we define a (real)
\emph{minimizing cycle} to be a cycle that realizes $|x|_\R$.

The goal of this subsection is to prove that minimizing cycles exist
(this is not obvious!), and to describe them in terms of multicurves.
In the next subsection, we will see that, for various points of
$\T(S)$, the associated spaces of minimizing cycles are polytopes that
exactly correspond to cells of $\B(S)$.

In what follows, integral minimizing cycles, rational minimizing
cycles, $|x|_\Z$, and $|x|_\Q$ are defined analogously to the real
case.

\p{Integral minimizing cycles.} We start by proving the necessary
facts about minimizing cycles in the integral case (see also
\cite{mr}), then proceed to the rational case, and finally the real
case.

\begin{lemma}
\label{lem:z exists}
Integral minimizing cycles exist, that is, $|x|_\Z$ is realized by an
integral cycle.  There are finitely many integral minimizing cycles.
\end{lemma}

\begin{proof}

Let $c$ be any integral cycle representing $x$.  There are finitely
many geodesics in $X$ with length less than or equal to $\ell_X(c)$,
and hence there are finitely many candidates for the support of an
integral minimizing cycle.  Given
any choice of support, there are finitely many choices of integral
coefficients that result in a cycle of length less than or equal to
$\ell_X(c)$.  The lemma follows.
\end{proof}

\begin{lemma}
\label{simple}
An integral minimizing cycle is simple.
\end{lemma}

\begin{proof}

Consider an integral cycle that is not simple, and assume its
self intersections are transverse.  By resolving the intersections as
in Figure~\ref{resolvepic}, we get a new, shorter, cycle in the class $x$.
\end{proof}

\begin{figure}[htb]
\centerline{\includegraphics[scale=.5]{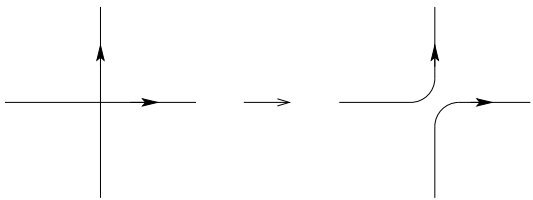}}
\caption{The resolution of an intersection.}
\label{resolvepic}
\end{figure}

We now want to prove that any two integral minimizing cycles have no
transverse intersections.  This will follow from the next lemma, which
appears, for instance, in work of Thurston \cite[Lemma 1]{wpt}.

\begin{lemma}
\label{lem:split}
Let $y$ be a nontrivial element of $H_1(S,\Z)$, and let $k$ be a positive integer.  Any simple integral 1-chain in $S$ representing $ky$ can be written as the sum of $k$ integral representatives for $y$.
\end{lemma}

\begin{proof}

Via Poincar\'e duality, we can think of $ky$ as an element of
$H^1(S,\Z)$, and as such, we obtain a map $f : S \to S^1$ with the
properties that $f^\star([S^1]) = ky$ and, for some regular value $p
\in S^1$, $f^{-1}(p)$ is equal to the given simple chain representing
$ky$.  Since the image of $f$ is $k\Z$, we can lift $f$ as follows
\begin{figure}[H]
\begin{center}
\scalebox{1}{ 
\xymatrix{ 
 & S^1 \ar@{->}[d]^{\pi} \\
S \ar@{->}[r]^{f} \ar@{->}[ur]^{\tilde f} & S^1
}
}
\end{center}
\end{figure}
where $\pi$ is multiplication by $k$.  Now $\pi^{-1}(p) =
\{p_1,\dots,p_k\}$ is a set of $k$ regular values for $\tilde f$, and by
construction, $\cup \tilde f^{-1}(p_i)=f^{-1}(p)$ is the original representative
for $ky$, and each $\tilde f^{-1}(p_i)$ is an integral cycle that
represents $y$.
\end{proof}

\begin{lemma}
\label{lem:qc}
For $q \in \Z$, we have $|qx|_\Z = q|x|_\Z$.
\end{lemma}

\begin{proof}

Let $c'$ be an integral minimizing cycle for $qx$.  By Lemma~\ref{simple}, $c'$ is simple, and so by Lemma~\ref{lem:split}, it splits into $q$ representatives of $x$.  Each of these representatives has length at least $|x|$, and so the lemma follows.
\end{proof}

\begin{lemma}
\label{integral disjoint}
Two integral minimizing cycles have no transverse intersections.
\end{lemma}

\begin{proof}

Suppose that two integral minimizing cycles for $x$ have transverse intersections.  By Lemma~\ref{lem:qc}, their sum would then be an integral minimizing cycle for $2x$ that is not simple, contradicting Lemma~\ref{simple}.
\end{proof}

\p{Rational minimizing cycles.}  The next lemma will help bridge
the gap between integral minimizing cycles and real minimizing cycles.

\begin{lemma}
\label{rational simple}
We have $|x|_\Q = |x|_\Z$; in particular, rational minimizing
cycles exist.  Also, any rational minimizing cycle is simple, and any two
rational minimizing cycles are disjoint.
\end{lemma}

\begin{proof}

Suppose $c$ is a rational cycle for $x$ with length less than
$|x|_\Z$.  Choose $q \in \Z$ so that $qc$ is integral.  Then $qc$ is
an integral cycle representing $qx$ with length less than
$q|x|_\Z$, contradicting Lemma~\ref{lem:qc}

Consider any rational minimizing cycle $c$ that is not simple.
Choose $q \in \Z$ so that $qc$ is integral.  By Lemma~\ref{lem:qc},
the cycle $qc$ is then an integral minimizing cycle for $qx$ that
is not simple, contradicting Lemma~\ref{simple}.

If we have two rational minimizing cycles that intersect, then we can take the ``average'' of the two cycles to get a rational minimizing cycle for
$x$ that is not simple, violating the second conclusion of the
lemma.
\end{proof}

\p{The borrowing lemma.}  The following lemma is the key technical
statement which will allow us to relate the complex of cycles to
the concept of a minimizing cycle.

\begin{lemma}[Borrowing lemma]
\label{lem:borrowing}
Let $c = \displaystyle \sum_{i=1}^N k_ic_i$ be a real positive cycle representing $x \in H_1(S,\Z)$.  Suppose that (after reindexing) there is a relation
\[ v_1 [c_1] + \cdots  + v_m [c_m] = v_{m+1} [c_{m+1}] + \cdots + v_{m+n} [c_{m+n}] \]
in $H_1(S,\R)$, where each $v_i$ is positive.  Let 
\[ L_1 = \sum_{i=1}^m v_i\ell_X(c_i) \ \ \mbox{ and } \ \ L_2 = \sum_{i=m+1}^{m+n}v_i\ell_X(c_i) \]
and assume (without loss of generality) that $L_1 \leq L_2$.  Let
$\delta$ be a positive real number less than or equal to $\min\{k_i/v_i\}$, where the minimum is taken over $m+1 \leq i \leq m+n$.  The cycle
\[ c_\delta = \sum_{i=1}^m (k_i+\delta v_i) c_i + \sum_{i=m+1}^{m+n} (k_i-\delta v_i) c_i + \sum_{i=m+n+1}^N k_ic_i \]
is a representative for $x$.  Furthermore, $\ell_X(c_\delta) = \ell_X(c)$ if and only if $L_1=L_2$ and $\ell_X(c_\delta) < \ell_X(c)$ if and only if $L_1 < L_2$.  Finally, the supports of $c$ and $c_\delta$ are equal if $\delta < \min\{k_i/v_i\}$.
\end{lemma}

The proof of the Borrowing lemma is a simple calculation, which we
leave to the reader.  As a basic example in genus 2, consider the
curves $a$, $b$, and $c$ in Figure~\ref{borrowpic}.  If the homology classes
are $x$, $y$, and $x-y$, respectively, and the lengths $\ell_a$,
$\ell_b$, and $\ell_c$ are close to zero and satisfy $\ell_a = \ell_b
+ \ell_c$, then the cycle $\epsilon a + (1-\epsilon)b+(1-\epsilon)c$ is
minimizing for any $\epsilon \in [0,1]$.  This interval's worth of minimizing cycles exactly corresponds to the edge in the left hand side of Figure~\ref{cellspic}.

\begin{figure}[htb]
\psfrag{x}{$a$}
\psfrag{y}{$b$}
\psfrag{z}{$c$}
\centerline{\includegraphics[scale=.5]{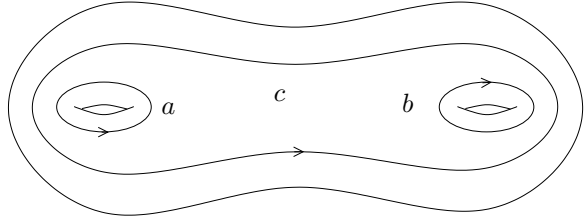}}
\caption{Curves in $S_2$ that give rise to ``borrowing''.}
\label{borrowpic}
\end{figure}

\p{Real minimizing cycles.}  Finally, we have the desired result about
real minimizing cycles.

\begin{lemma}
\label{real simple}
We have $|x|_\R = |x|_\Q = |x|_\Z$; in particular, real
minimizing cycles exist.  Also, any real minimizing cycle is simple,
and any two real minimizing cycles are disjoint.
\end{lemma}

\begin{proof}

Let $c = \sum k_ic_i$ be a nonrational real cycle.  Since $x$ is rational, it follows that the set of $c_i$ with $k_i$ irrational are linearly dependent.  An application of the Borrowing lemma allows us to reduce the number of irrational coefficients without increasing the length of the representative.  In the end, we get a rational cycle.  Thus, $|x|_\R = |x|_\Q$, and these are equal to $|x|_\Z$ by Lemma~\ref{rational simple}.

If $c$ is a real cycle that is not simple, we can, as above, use the Borrowing lemma to get a rational cycle of the same length that is also not simple.  Since rational minimizing cycles are simple (Lemma~\ref{rational simple}), it follows that $c$ is not a real minimizing cycle.

The proof of the last statement follows similarly.
\end{proof}


\subsection{The space of minimizing cycles}
\label{section:borrowing cells}

The goal of this subsection is to describe the parameter space of
minimizing cycles for any one point of $\T(S)$.  We will see that
each such parameter space corresponds to a cell of $\B(S)$.
As above, $x \in H_1(S,\Z)$ and $X \in \T(S)$ are fixed.


\p{Minimizing multicurves.}  A consequence of Lemma~\ref{real simple}
is that there is a canonical multicurve $M=M_X$ in $S$, called the
\emph{minimizing multicurve}, which is the smallest multicurve with
the property that every minimizing cycle for $x$ is supported on $M$.
The minimizing multicurve $M$ is obtained by taking
the union of the supports of all minimizing cycles for $x$.  Note that
$M$ is a union of geodesics in $X$.

\begin{lemma}
\label{lemma:orient}
There is a unique orientation of the minimizing multicurve $M=M_X$ so
that each minimizing cycle for $x$ is a nonnegative linear combination
of curves of $M$.
\end{lemma}

\begin{proof}

Suppose $M$ is the collection of oriented curves $\{c_1,\dots, c_n\}$,
and that $c = \sum k_ic_i$ and $c'=\sum k_i'c_i$ are minimizing cycles
for $x$ with $k_i' < 0 < k_i$ for some particular $i$.  In this case,
$c+c'$ is a representative for $2x$ with length strictly less than
twice that of $c$.  This contradicts Lemma~\ref{lem:qc}.  
\end{proof}

As a consequence of the lemma, it makes sense to focus our attention on
positive minimizing cycles for $x$.


\p{Borrowing relations.}  Let $M=\{c_i\}$ be an oriented multicurve.
With the Borrowing lemma in mind, we call any relation amongst the
$[c_i]$ in $H_1(S,\Z)$ a \emph{borrowing relation}.  If we cut $S$
along $M$, then each resulting connected component gives rise to a
borrowing relation, which we call a \emph{subsurface borrowing
  relation}.

\begin{lemma}
\label{lemma:borrowing relations}
Let $X \in \T(S)$, and let $M=\{c_i\}$ be an oriented multicurve that
is the union of (oriented) supports of positive cycles that represent
$x$.  The following are equivalent.
\begin{enumerate}
\item All positive cycles for $x$ supported in $M$ have the same length.
\item For every subsurface borrowing relation $\sum v_i [c_i]=0$, the
  $\ell_X(c_i)$ satisfy $\sum v_i \ell_X(c_i)=0$.
\end{enumerate}
In particular, if $M$ is a minimizing multicurve, then both conditions
are satisfied.
\end{lemma}

\begin{proof}

By the Borrowing lemma, statement (1) is equivalent to the statement:
(3)~for every borrowing relation $\sum v_i [c_i]=0$, the $\ell_X(c_i)$
satisfy $\sum v_i \ell_X(c_i)=0$.  By basic algebraic topology, any
borrowing relation between the $[c_i]$ is a consequence of the
subsurface borrowing relations between the $[c_i]$.  Hence (3) is
equivalent to statement (2).
\end{proof}


\p{Borrowing cells.}  We now explain how the minimizing multicurve $M=M_X$ gives a
natural way to parameterize the space of minimizing cycles in $X$.  As
in Lemma~\ref{lemma:orient}, we orient $M$ so that all minimizing
cycles are positive.

As in Section~\ref{section:complex}, $D=D(M)$ is the dimension of the
span of $M$ in $H_1(S,\R)$ and $B=B(M)$ is $|M|-D$.  There is a
natural map $\R^{|M|} \to H_1(S,\R)$, and the preimage of the subspace
spanned by $x$ is a $B$-dimensional affine subspace $V_X$ of
$\R^{|M|}$.  By taking lengths of cycles in $X$, we get a map $L_X :
V_X \to \R_+$.

We can think of the minset $P_X$ of $L_X$ in $V_X$ as the parameter
space of minimizing cycles for $x$ in $X$.  As $P_X$ is the minset of a sum
of convex functions it is a convex set $P_X$, which we call a
\emph{borrowing cell}.  Since each of these (finitely many) functions
is proper, it follows that $P_X$ is a compact.  Finally, since each of
the functions is the absolute value of a linear function, it follows
that $P_X$ is a polytope (it is cut out by the zero sets of the
individual functions).

\begin{lemma}
Each borrowing cell $P_X$ is a compact convex polytope in $V_X$.
\end{lemma}

We now prove that the borrowing cell $P_X$ is the convex hull of its
basic cycles for $x$.  This gives the key connection between
minimizing cycles and the complex $\B(S)$.  Recall from
Section~\ref{section:complex} that $P_M$ is the cell of $\B(S)$
associated to a multicurve $M$.

\begin{lemma}
\label{lemma:vertices vs multicurves}
There is a bijective correspondence between the set of basic
cycles for $x$ that are supported in $M=M_X$ and the set of
vertices of $P_X$.  In particular, $P_X=P_{M_X}$.
\end{lemma}

\begin{proof}

We use the following characterization: a point of a polytope in
Euclidean space is a vertex if and only if every open line segment
through that point leaves the polytope.

First, let $p$ be a point of $P_X$ corresponding to a basic minimizing
cycle.  If we move away from $p$ in any direction of $V_X$, some
coordinate must change from zero to nonzero, and it follows that, in
any open line segment of $V_X$ through $p$, there must be a point with
at least one negative coordinate.  By Lemma~\ref{lemma:orient}, such
points are not contained in $P_X$, and so $p$ is a vertex.

Next, let $p$ be a point of $P_X$ that corresponds to a minimizing
cycle that is not basic.  In this case, there is an open line segment
of $V_X$ through $p$ so that all points on the line segment have
positive coordinates.  Since $M$ is a minimizing cycle, it follows
from the last sentence of Lemma~\ref{lemma:borrowing relations} that
statement (1) of the same lemma holds, and we see that the line segment
is contained in $P_X$.  Thus, $p$ is not a vertex.
\end{proof}

The reader is invited to revisit the examples of
Section~\ref{section:complex}, and reinterpret them as borrowing cells
in the context of minimizing cycles.

Recall from Section~\ref{section:complex} that $\M$ is the category of
oriented multicurves that are made up of basic cycles for $x$, and
$\F$ is the category of cells of $\B(S)$.  The set of borrowing cells
also forms a category under inclusion, if we view each borrowing cell
as a subset of $\R^\S$.

\begin{lemma}
\label{lemma:covariant}
The following categories are isomorphic: $\M$, $\F$, and the category of borrowing cells.
\end{lemma}

\begin{proof}

The equivalence of $\M$ and $\F$ was already
stated in Section~\ref{section:complex}.  Also, by
Lemma~\ref{lemma:vertices vs multicurves}, the functor $P_X \mapsto
M_X$ is an isomorphism from the category of borrowing cells to its
image in the category $\M$.  It remains to show that this functor is
surjective.

In other words, we need to produce, for a given $M \in \M$, a marked
hyperbolic surface $X \in \T(S)$ with $M_X=M$.  This is done by
choosing the lengths of the curves of $M$ to be much smaller than all
other curves in $X$, and inductively choosing the lengths of the
curves of $M$ to satisfy the subsurface borrowing relations, as in
Lemma~\ref{lemma:borrowing relations}.
\end{proof}


\subsection{Chambers}
\label{section:chambers}

To each borrowing cell (equivalently, to each element of $\M$), we now
associate a subset of $\T(S)$, which we call a ``chamber''.  We will
show that these chambers are contractible
(Proposition~\ref{prop:chamber contract}), and that the cover of
$\T(S)$ by chambers is, in a sense, dual to the cover of $\B(S)$ by
cells.  Since $\T(S)$ is contractible, it will follow that $\B(S)$ is
contractible.

Let $M \in \M$.  We define the \emph{closed
  chamber} associated to $M$ as
\[ Y_M = \{X \in \T(S) : M \subseteq M_X \}. \]

The following proposition is proved in Section~\ref{section:chambers},
and is the key technical part of the paper.

\begin{prop}
\label{prop:chamber contract}
Each closed chamber $Y_M$ is contractible.
\end{prop}

The next lemma is used in the proof of Proposition~\ref{prop:chamber
  contract}.  See the appendix of \cite{bs} for an introduction to
manifolds with corners.  For the statement, we define the \emph{open
  chamber} associated to the oriented multicurve $M$ as follows:
\[ Y_M^o = \{ X \in \T(S) : M = M_X  \}. \]
The proof of the lemma uses the Fenchel--Nielsen coordinates for
$\T(S)$ (see \cite[Section 3.2]{it}), and the following basic fact
from Teichm\"uller theory.

\begin{fact}
\label{fact:rls}
Let $X_0 \in \T(S)$, and $D>0$.  There is a neighborhood $U$ of $X_0$ so
that the set $\{c \in \S : \ell_X(c) < D \mbox{
  for some } X \in U \}$ is finite.
\end{fact}

\begin{lemma}
\label{lem:chamber with corners}
Each closed chamber $Y_M$ is a manifold with corners.  Its interior is
the open chamber $Y_M^o$.
\end{lemma}

\begin{proof}

Let $X \in Y_M$.  Choose a pants decomposition for $S$ containing
the curves of $M$, and consider the resulting Fenchel--Nielsen
coordinates on $\T(S)$.  By Fact~\ref{fact:rls}, there is a finite set of linear
inequalities among the length coordinates, so that, in a neighborhood
$U$ of $X$, a point is in $Y_M$ if and only if those inequalities
are satisfied.  Since the solution set of a system of linear
inequalities in $\R^n$ is a manifold with corners, and
Fenchel--Nielsen coordinates form a smooth chart for $\T(S)$, it
follows that $Y_M$ is a manifold with corners.

In the previous paragraph, the locus of points in $\R^n$ where all of the inequalities are strict is exactly where $M$ is the minimizing multicurve; but this is the definition of $Y_M^o$.
\end{proof}

We now relate the partial orderings on multicurves and closed chambers.

\begin{lemma}
\label{lem:face relations}
$Y_M \supseteq Y_{M'}$ if and only if $M \subseteq M'$.
\end{lemma}

\begin{proof}

By the definition of closed chambers, the statement $Y_M \supseteq
Y_{M'}$ is equivalent to the statement that, for any $X \in \T(S)$, if
$M_X \supseteq M'$, then $M_X \supseteq M$.  The latter is equivalent
to the inclusion $M \subseteq M'$.
\end{proof}

Given two multicurves $M$ and $M'$ with no transverse intersections,
the multicurve $M \cup M'$ is defined by taking the union of the two
multicurves and replacing each pair of isotopic curves with a single
curve.

\begin{cor}
\label{cor:chamber int}
The intersection of two closed chambers $Y_M$ and $Y_{M'}$ is nonempty
if and only if the multicurve $M \cup M'$ is defined and has an
orientation compatible with those of $M$ and $M'$.  In this case, $Y_M
\cap Y_{M'} = Y_{M \cup M'}$.
\end{cor}

From the above corollary we can deduce that the category of closed
chambers (under inclusion), and the opposite of the category $\M$ are
isomorphic.  We will not use this fact directly, but the
contravariance belies the main difficulty in the proof of
Theorem~\ref{main:b} in Section~\ref{section:contractible}.

\begin{cor}
\label{cor:loc fin}
The nerve of the cover of $\T(S)$ by the closed chambers $\{Y_M\}$ is finite dimensional.
\end{cor}

\begin{proof}

Suppose $\{Y_{M_i}\}$ is a collection of closed chambers with nontrivial intersection.  By Corollary~\ref{cor:chamber int}, this intersection is a closed chamber $Y_M$, and each $M_i$ is a submulticurve of $M$.  Now, $M$ consists of at most $3g-3$ curves, and so the set $\{Y_{M_i}\}$ contains at most $2^{3g-3}$ elements.  This gives an upper bound for the dimension of the nerve.
\end{proof}


\subsection{The complex of cycles is contractible}
\label{section:contractible}

In this section, we prove Theorem~\ref{main:b} by making a comparison
between the cover of $\B(S)$ by cells and the cover of $\T(S)$ by
chambers.  Roughly, since both chambers and cells of $\B(S)$ are
contractible, and since the gluing patterns are the same (up to
contravariance), it will follow that $\B(S)$ is homotopy equivalent to
$\T(S)$, and so $\B(S)$ is contractible.

\begin{proof}[Proof of Theorem~\ref{main:b}]

Let $\B'(S)$ denote the barycentric subdivision of $\B(S)$; it suffices to show that $\B'(S)$ is contractible.  For a vertex of $\B(S)$ corresponding to a multicurve $M$, denote by $C_M$ the star of the corresponding vertex in $\B'(S)$.

The proof is in three steps.
\begin{enumerate}
\item The geometric realization of the nerve of the cover of $\T(S)$ by closed chambers is homotopy equivalent to $\T(S)$.
\item The geometric realization of the nerve of the cover of $\B'(S)$ by the $\{C_M\}$ is homotopy equivalent to $\B'(S)$.
\item The two nerves are the same.
\end{enumerate}

\emph{Step 1.} It is a special case of a result of Borel--Serre that if a
topological space is covered by a set of closed, contractible
manifolds with corners, and the dimension of the nerve of this cover
is finite, then the nerve is homotopy equivalent to the original space
\cite[Theorems 8.2.1, 8.3.1]{bs}.  Thus, it suffices to apply Proposition~\ref{prop:chamber contract},
Lemma~\ref{lem:chamber with corners}, and Corollary~\ref{cor:loc
fin}.

\emph{Step 2.} We apply the same theorem of Borel--Serre.  In general,
the closed star of a vertex of a simplicial complex is contractible
($\B(S)$ is not simplicial, but its barycentric subdivision is).  In
the category of simplicial complexes, the phrase ``closed manifold with
corners'' used in Step 1 can be replaced with ``subcomplex'' (both types of
spaces are examples of absolute neighborhood retracts).  That the dimension of the nerve is
finite follows from Step 3 (which is independent of the other steps).
Thus, their theorem applies.

\emph{Step 3.} A collection $\{C_{M_i}\}$ has nonempty intersection if
and only if the vertices $M_i$ are vertices of a common cell of
$\B(S)$.  By Lemma~\ref{lemma:covariant}, this is equivalent to the
condition that the $M_i$ are vertices of some common borrowing cell,
that is, if there is a hyperbolic structure $X$ with $M_i \subseteq
M_X$ for all $i$.  But this is the same as the condition that the
intersection of the $Y_{M_i}$ is nonempty (Corollary~\ref{cor:chamber
int}).

Combining Steps 1, 2, and 3, plus the fact that $\T(S)$ contractible,
we see that $\B'(S)$, and hence $\B(S)$, is contractible.
\end{proof}


\section{Chambers are contractible}
\label{section:chambers}

In this section we prove Proposition~\ref{prop:chamber contract},
which says that each closed chamber $Y_M$ is contractible.  We take
$x$ to be the element of $H_1(S,\Z)$ chosen in Section~\ref{section:complex}.
As in Section~\ref{section:borrowing complex}, $Y_M$ is the closed chamber
associated to some multicurve $M$ whose isotopy class is in $\M$.

\subsection{The thick part of Teichm\"uller space is contractible}

The goal is to reduce Proposition~\ref{prop:chamber contract} to the
following theorem of Ivanov \cite[Theorem 3]{nvi3}.  The
\emph{$\epsilon$-thick part} of Teichm\"uller space $\T_\epsilon(S)$
is the subset consisting of hyperbolic surfaces where every
nontrivial closed geodesic that is not a boundary component has length
at least $\epsilon$ (in the case of a surface with boundary, we define
Teichm\"uller space as the space of marked hyperbolic surfaces where
each boundary is a geodesic of fixed length).

Also, we make use of the Margulis constant $\epsilon_{\hyp^2}$ for the
hyperbolic plane.  This number has the property that, in any
hyperbolic surface, two nonisotopic curves of length less than $\epsilon_{\hyp^2}$ are disjoint.

\begin{theorem}
\label{thm:thick contract}
Let $S$ be any compact surface and let $\epsilon \in (0,\epsilon_{\hyp^2})$.  The space $\T_\epsilon(S)$ is contractible.
\end{theorem}

For any $\epsilon \in (0,\epsilon_{\hyp^2})$ and $p \in \R_+^{|M|}$,
we define the subset $\A_{\epsilon,p}$ of $\T(S)$ to be the set of
points with the following properties.
\begin{enumerate}
\item Each curve in $S$ that is disjoint from the curves of $M$ and
  is not isotopic to a curve of $M$ has length at least $\epsilon$.  
\item The lengths of the curves of $M$ are given by $p$.
\end{enumerate}

\begin{lemma}
\label{lemma:choose}
For any $\epsilon < \epsilon_{\hyp^2}$, there is a choice of $p$ so that $\A_{\epsilon,p}$ is a subset of the open chamber $Y_M^o$.
\end{lemma}

\begin{proof}

Let $\epsilon < \epsilon_{\hyp^2}$ be given.  We choose $p \in
\R_+^{|M|}$ so that every positive cycle for $x$ supported (with
orientation) in $M$ has the same length less than $\epsilon$ (take any
point in $Y_M^o$, and shrink the curves of $M$ by the same factor until this
condition holds).

In $\A_{\epsilon,p}$, a basic minimizing cycle for $x$ cannot have
support on any curve which is disjoint from (and not isotopic to) $M$,
since all such curves have length at least $\epsilon$ (by the
definition of $\A_{\epsilon,p}$).  Also, no geodesic intersecting the
geodesic representative of $M$ transversely can be in the support of a
basic minimizing cycle, since all such curves have length greater than
$\epsilon$ by the hypothesis on $\epsilon$.  Thus, $\A_{\epsilon,p}
\subset Y_M^o$.
\end{proof}

\begin{lemma}
\label{lemma:a ep}
For any $\epsilon < \epsilon_{\hyp^2}$ and any $p$, the space $\A_{\epsilon,p}$ is contractible.
\end{lemma}

\begin{proof}

Take a pants decomposition $P$ of $S$ containing $M$, and consider the
associated Fenchel--Nielsen coordinates.

Let $\{S_i\}$ be the set of connected components of $S-M$.  There is a
natural map
\[ \A_{\epsilon,p} \to \T_\epsilon(S_i) \]
for each $i$, since each $S_i$ has an induced marking and an induced hyperbolic structure (we make sure to cut along the geodesic representative of $M$).

We also consider the $|M|$ maps
\[ \A_{\epsilon,p} \to \R \]
obtained by taking the twist parameters at the curves of $M$.

Combining these maps, we get a bijective map
\[ \A_{\epsilon,p} \to \prod \T_\epsilon(S_i) \times \R^{|M|}.\]
This map is a homeomorphism---there is a continuous inverse, obtained by gluing the pieces back together according to the prescribed twist parameters.
Thus, applying Theorem~\ref{thm:thick contract}, we see that
$\A_{\epsilon,p}$ is contractible.
\end{proof}

\subsection{Chamber flow}

The goal of this section is to prove the following technical
statement, which, combined with Lemmas~\ref{lemma:choose}
and~\ref{lemma:a ep}, completes the proof that closed chambers are
contractible (Proposition~\ref{prop:chamber contract}).

\begin{lemma}
\label{lem:flow1}
Given any compact subset $K$ of the closed chamber $Y_M$, there is a
choice of $\A_{\epsilon,p} \subset Y_M^o$, and a deformation of a
subset of $Y_M$ which takes $K$ into $\A_{\epsilon,p}$.
\end{lemma}

For the proof, we need two lemmas from Teichm\"uller theory.  For a
pants decomposition $P$ of a surface $S=S_g$, let $FN_P$ be the map
$\R^{6g-6} \to \T(S)$ corresponding to the Fenchel--Nielsen
coordinates adapted to $P$, and let $FN_P^\star$ be the differential.
In the statement of the lemma, we endow $\R^{6g-6}$ with the Euclidean
metric, and $\T(S)$ with the Teichm\"uller metric.

\begin{lemma}
\label{lem:lipschitz}
Let $S=S_g$ and $\Dmin, \Dmax \in \R_+$.  There is a constant
$C=C(S,\Dmin,\Dmax)$ so that if $P$ is any pants decomposition of $S$, and
\[ W = \{X \in \T(S) : \Dmin \leq \ell_X(c) \leq \Dmax \mbox{ for all } c \in P\}, \]
then $FN_P^\star$ is $C$-Lipschitz on $FN_P^{-1}(W)$.
\end{lemma}

\begin{proof}

Let $P$ be given.  The subset $FN_P^{-1}(W)$ of $\R^{6g-6}$ is the
infinite ``strip''
\[ [\Dmin,\Dmax]^{3g-3} \times \R^{3g-3}, \]
that is, all length parameters are in $[\Dmin,\Dmax]$, and twist parameters are unconstrained.  The free abelian group $G(P)$ generated by Dehn twists in the curves
of $P$ acts on $W$, since $G(P)$ preserves the lengths of the curves
of $P$.  Moreover, this action is by isometries (in general, $\Mod(S)$
acts on $\T(S)$ by isometries).

In the Fenchel--Nielsen coordinates, the Dehn twist about a curve of $P$ is translation by $2\pi$ in the corresponding twist coordinate.  Therefore, the quotient of $FN_P^{-1}(W)$ by $G(P)$ is 
\[ [\Dmin,\Dmax]^{3g-3} \times (S^1)^{3g-3}. \]
Again, this action is by isometries, so the stretch factor of
$FN_P^\star$ is determined on this quotient.  As this quotient is
compact, and $FN_P$ is smooth, it follows that $FN_P^\star$ is
Lipschitz on $FN_P^{-1}(W)$.

To get the Lipschitz constant $C$ which only depends on the surface
$S$ and the constants $\Dmin$ and $\Dmax$ (and not on the pants
decomposition $P$), we note that there are only finitely many
topological types of pants decompositions on $S$.  Since $\Mod(S)$
acts by isometries on $\T(S)$, and since $\Mod(S)$ preserves the
infinite strips $[\Dmin,\Dmax]^{3g-3} \times \R^{3g-3}$, the Lipschitz
constant is the same for two pants decompositions which are
topologically equivalent.  Therefore, the constant
$C=C(S,\Dmin,\Dmax)$ is the maximum of finitely many Lipschitz
constants obtained as above.
\end{proof}

The next lemma is the technical statement which implies the existence
of the Bers constant---the universal constant $L_g$ with the property
that every hyperbolic surface of genus $g$ has a pants decomposition
of total length at most $L_g$ (cf. the inductive step of \cite[Theorem 5.2.3]{pb}).

\begin{lemma}
\label{lem:complete pd}
Let $g \geq 2$ and $L > 0$.  There is a constant $D=D(g,L)$ so that, if $M$ is a
geodesic multicurve of $X \in \T(S_g)$ where each curve of $M$ has
length at most $L$, then there is a pants decomposition $P$
containing $M$ so that each curve of $P$ has length less than $D(g,L)$.
\end{lemma}

We now construct the deformation of $K$ into $\A_{\epsilon,p}$.  The argument is modelled on Ivanov's proof of Theorem~\ref{thm:thick contract}.

\begin{proof}[Proof of Lemma~\ref{lem:flow1}]

We will define a smooth vector field $\V$ on an open subset of $Y_M$, and the resulting flow will give the required deformation.  The
basic idea is to use Fenchel--Nielsen coordinates to define a vector field that, at a given point, points in a direction where the following conditions hold.
\begin{list}{\boldmath$\cdot$}{}
\item The curves of $M$ are getting shorter, and their lengths are approaching the point $p \in \R_+^{|M|}$.  
\item Any curve that is not in $M$, is disjoint from $M$, and has
  length $\epsilon$, stays the same length; in particular, no curve
  outside of $M$ gets shorter than $\epsilon$ (we will choose $\epsilon > 0$ so that all curves begin with length greater than $\epsilon$).
\item Any minimizing cycle for $x$ not supported in $M$ has at least
  one curve that stays the same length (in particular, by the first condition,
  minimizing cycles not supported in $M$ immediately become
  nonminimizing).
\end{list}
In short, vectors should point towards the interior of $Y_M$ (third condition), and in particular towards $\A_{\epsilon,p}$ (first two conditions).  To accomplish this, there are three steps.

\hspace{.25in}Step 1: Make a choice of $\A_{\epsilon,p}$.

\hspace{.25in}Step 2: Define the vector field $\V$ locally.

\hspace{.25in}Step 3:  Check that $\V$ has the desired properties.

\emph{Step 1.} Picking an $\A_{\epsilon,p}$ involves choosing
$\epsilon$ and $p$.  Let $L_K$ be the length of the longest minimizing
cycle for $x$ in $K$.  We choose $\epsilon < \epsilon_{\hyp^2}$ small enough so that the
following two conditions are satisfied.
\begin{list}{\boldmath$\cdot$}{}
\item Every nontrivial geodesic in each surface of $K$ has length at
  least $\epsilon$. 
\item If $c$ is a geodesic that crosses another geodesic of length at
  most $\epsilon$, then $c$ has length greater than $L_K$.
\end{list}
To see that both conditions can be satisfied for $\epsilon$ small
enough, we use the fact that the length of a particular curve in $S$
is a continuous function on $\T(S)$, and for the second condition, we
apply the collar lemma, which states that, in any hyperbolic surface,
as the length of a geodesic $d$ tends to zero, the length of the
shortest geodesic intersecting $d$ transversely tends to infinity.

The point of the first condition on $\epsilon$ is that, in order to
prevent curves from getting shorter than $\epsilon$, they must be
longer than $\epsilon$ to begin with.  The second condition implies
that, in $K$, and at any point of $Y_M$ where the length of a
minimizing cycle is at most $L_K$, curves of length $\epsilon$
(or shorter) have trivial geometric intersection with the curves of
$M$.

In the proof of Lemma~\ref{lemma:choose}, the choice of $p$ only
depended on the choice of $\epsilon$ and on $M$.  We choose $p$ in the
same way here, using the $\epsilon$ defined here.  As a result, we have $\A_{\epsilon,p} \subset Y_M^o$.

\emph{Step 2.} We will define the vector field $\V$ on open sets
indexed by pants decompositions.  First, let $\Dmax$ be the constant
$D(g,L_K)$ from Lemma~\ref{lem:complete pd}, where $L_K$ is the
constant from Step 1.  Note that $\Dmax > L_K$.  Choose $\Dmin$ to
be smaller than the smallest coordinate of $p$ in $\R^{|M|}$.  Note
that $\Dmin < |p| < \epsilon$, and, in particular, $\Dmin$ is smaller
than the length of any curve in any surface of $K$.

Let $\P$ be the set of pants decompositions of $S$ that contain $M$.
For $P \in \P$, let $U_P$ be the open subset of $Y_M$ given by
\begin{eqnarray*}
U_P = \{ X \in Y_M & : & \mbox{if } X \in Y_{M'} \mbox{ then } M'
\subseteq P, \\
&& \mbox{if } \ell_X(c) \leq \epsilon \mbox{ then } c \in P, \\
&& \Dmin < \ell_X(c) < \Dmax \mbox{ for all } c \in P \}.
\end{eqnarray*}

We remark that the last two conditions in the definition of $U_P$ are
not contradictory since $\Dmax$ is automatically at least $\epsilon$.
Also, the second condition is indeed an open condition: restate it as
$c \notin P \Rightarrow \ell_X(c) > \epsilon$ and apply Fact~\ref{fact:rls}.

Let $U=\cup_{P \in \P} U_P$.  The compact set $K$ is
contained in $U$; this follows from the choices of $\Dmax$ and
$\Dmin$, Lemma~\ref{lem:complete pd}, and the fact that there are no
curves of length less than or equal to $\epsilon$ in surfaces of $K$.

For a given pants decomposition $P$, we now define a vector field
$\V_P$ on $U_P$.  Let $d: \R^{|M|} \to \R_{\geq 0}$ be the square of
the Euclidean distance from the point $p \in \R^{|M|}$.  Write
$\R^{6g-6}$ as $\R^{|M|} \times \R^{6g-6-|M|}$, and let $\V_P'$ be the
vector field on $FN_P^{-1}(U_P) \subset \R^{6g-6}$ given by $(-\nabla
d,0)$.  We then define $\V_P$ on $U_P$ via $\V_P = FN_P^\star(\V_P')$.

Consider the map
\[ \Phi_P: U_P \to \R^{|M|} \]
obtained by taking the lengths of the curves of $M$.  By construction,
$\Phi_P^\star(\V_P) = -\nabla d$.

There is a partition of unity subordinate to the cover of $U$ by the
$U_P$ which we use to combine the local vector fields $\V_P$ into a
global vector field $\V$ defined on all of $U$.  Since $\V$ is
obtained by averaging the $\V_P$, we still have $\Phi^\star(\V) =
-\nabla d$, where $\Phi : U \to \R^{|M|}$ is the map which records the
lengths of the curves of $M$.  In other words, even after using the
partition of unity, we still know what happens to the lengths of the
curves of $M$ along flow lines of $\V$.

\emph{Step 3.}  We will prove the following statements.

\begin{enumerate}
\item All flow lines starting in $K$ stay in $Y_M$.
\item All flow lines starting in $K$ stay in $U$.
\item There is a universal constant $C$ so that $||v|| \leq C||\Phi^\star(v)||$ for all $v$.
\item All flow lines starting in $K$ converge as time goes to
  infinity.
\end{enumerate}

The first three statements imply that the flow induced by $\V$ is
defined for all positive time.  From the fourth statement (plus the
definition of $\V$), we see that, at infinity, all points of $K$ end
up in $\A_{\epsilon,p}$.  The resulting map $K \to \A_{\epsilon,p}$ is
automatically continuous.  As $\A_{\epsilon,p}$ is contractible
(Lemma~\ref{lemma:a ep}), the lemma will follow.

To prove statement (1), there are two things to show: the
positive cycles for $x$ supported in $M$ all remain the same length,
and no other cycles for $x$ get shorter than those.  The first holds
because the subset of $\R^{|M|}$ where the positive cycles for $x$
have the same length is convex (by Lemma~\ref{lemma:borrowing
  relations} it is cut out by the linear equations coming from the
subsurface borrowing relations); since we are flowing along straight lines
towards $p$, we stay in this convex set.  For the second, we note
that, by construction, all vectors on the boundary of $Y_M$ point into
the interior of $Y_M$, i.e. the open chamber $Y_M^o$.

For statement (2), let $X$ be a point on a flow line which emanates from
$K$.  By the construction of $\V$, the image of $Y_M$ under the flow
induced by $\V$ is contained in $Y_M^o$ for all positive time.
Therefore, since we already showed that $K \subset U$, we may assume
that $X \in Y_M^o$.  We need to argue that there is a pants
decomposition $P$ of $X$ so that $X \in U_P$.  To construct $P$, we
start with $M$, add all curves of length at most $\epsilon$, and (if
needed) complete to a pants decomposition using curves of length at
most $\Dmax$.  We now check that $X \in U_P$; the only nontrivial
point is the condition $\Dmin < \ell_X(c) < \Dmax$ for all $c \in P$.
Each $\ell_X(c)$ is less than $\Dmax$ by the definition of $\Dmax$ and
the fact that the lengths of the curves of $M$ get shorter along flow
lines.  Each $\ell_X(c)$ is greater than $\Dmin$ because each curve of
$M$ stays longer than $|p| > \Dmin$ along flow lines, and
curves disjoint from $M$ do not get shorter than $\epsilon > \Dmin$
along flow lines.

For statement (3), we apply Lemma~\ref{lem:lipschitz}.  By the
definition of $U_P$, for any $P \in \P$, and any curve $c$ of $P$, we
have $\Dmin < \ell_X(c) < \Dmax$.  Let $C=C(S,\Dmin,\Dmax)$ be the
constant given by Lemma~\ref{lem:lipschitz}.  By the lemma, each map
$FN_P : FN_P^{-1}(U_P) \to U_P$ is $C$-Lipschitz.  In other words, the
length of a vector of $\V_P$ is at most $C$ times the length of the
corresponding vector of $\V_P'$.  The latter is the same as the length
of the corresponding vector of $-\nabla d$.  Since $\V$ is obtained
from the $\V_P$ by averaging, it follows that each vector of $\V$ has
length at most $C$ times the length of the corresponding vector of
$-\nabla d$ (the correspondence is given by $\Phi^\star$), which is
what we wanted to show.

From statement (3), it follows that if $\gamma$ is an integral path in
$\R^{6g-6}$ of length $L$, then the lift of this path to $\T(S)$ has
length at most $CL$.  Therefore, since flow lines of $\V_P'$ in
$\R^{6g-6}$ all converge, the corresponding flow lines in $\T(S)$ also
converge (statement (4)).  This completes the proof.
\end{proof}


\section{Stabilizer dimension}
\label{section:stabilizers}

Recall from the Introduction that, to prove
Theorems~\ref{main:torelli} and~\ref{main:kg}, it suffices to show, for each cell $\sigma$ of the complex of
cycles $\B(S_g)$, that the following inequalities hold.
\begin{eqnarray*}
\cd(\Stab_{\I(S_g)}(\sigma)) + \dim(\sigma) &\leq& 3g-5 \\
\cd(\Stab_{\K(S_g)}(\sigma)) + \dim(\sigma) &\leq& 2g-3
\end{eqnarray*}

By Lemma~\ref{lemma:covariant}, a cell $\sigma$ of $\B(S_g)$ is
associated to a multicurve $M$ consisting of nonseparating curves.
The stabilizer in $\I(S_g)$ of the cell $\sigma$ is exactly the
stabilizer of $M$.  Since an element of $\I(S_g)$ cannot reverse the
orientation of a nonseparating curve, the orientation of $M$ will play
no role in this section.

As in Section~\ref{section:complex}, the dimension of the cell
$\sigma$ is the number $B(M)$, which only depends on $M$ (without orientation).

In summary, the proof of Theorem~\ref{main:torelli} is reduced to the following.

\begin{prop}
\label{prop:dim lemma}
For $g \geq 2$ and $M$ as above, $\cd(\Stab_{\I(S_g)}(M)) + B(M) \leq 3g-5$. 
\end{prop}

And Theorem~\ref{main:kg} is implied by the
following proposition.

\begin{prop}
\label{prop:dim lemma kg}
For $g \geq 2$ and $M$ as above, $\cd(\Stab_{\K(S_g)}(M)) + B(M) \leq 2g-3$. 
\end{prop}

Before proving Propositions~\ref{prop:dim lemma} and~\ref{prop:dim lemma kg} in
Sections~\ref{section:proof of dim lemma} and~\ref{section:kg},
respectively, we need some preliminaries related to the Birman exact sequence.



\subsection{The Birman exact sequence and dimension}
\label{section:bes}

In order to state the classical Birman exact sequence, we need one definition.  The \emph{pure mapping class group} of a surface $S$, denoted $\PMod(S)$, is the subgroup of $\Mod(S)$ consisting of elements that fix each puncture and boundary component of $S$ individually.

\begin{theorem}[Birman exact sequence]
\label{bes}Let $S$ be a surface of negative Euler characteristic, and let $S'$ be
$S-p$.  We have
\[ 1 \to \pi_1(S,p) \to \PMod(S') \to \PMod(S) \to 1. \]
\end{theorem}
The map $\PMod(S') \to \PMod(S)$ is obtained by ``forgetting'' the
puncture $p$, and the image of an element of $\pi_1(S)$ in $\PMod(S')$
is realized by ``pushing'' the puncture $p$ around that element of
$\pi_1(S)$; see \cite[Section 4.1]{jsb}.

We will use the Birman exact sequence to describe the effect of
punctures on the cohomological dimension of subgroups of the mapping
class group.  There are three statements; one for higher genus ($g
\geq 2$), one for genus 1, and one for genus 0.  We use $S_{g,p}$ to
denote a surface of genus $g$ with $p$ punctures (and no boundary).

The three corollaries below follow, in a straightforward way, by inducting on the number of punctures, and applying the following facts; see \cite[Chapter VIII, Proposition 2.4]{ksb}.

\begin{fact}[Subadditivity of cohomological dimension]
\label{fact:cd ses}
If we have a short exact sequence of groups
\[ 1 \to K \to G \to Q \to 1 \]
then $\cd(G) \leq \cd(K) + \cd(Q)$.
\end{fact}

\begin{fact}[Monotonicity of cohomological dimension]
\label{fact:cd subgp}
If $H$ is a subgroup of $G$, then $\cd(H) \leq \cd(G)$.
\end{fact}

One also needs the basic facts that $\cd(\pi_1(S_g)) = 2$ for $g \geq
1$, and $\cd(\pi_1(S_{g,p})) = 1$ if $p>0$, and $(g,p) \neq (0,1)$.

We start with the case $g \geq 2$.  For the application of
Corollary~\ref{cor:cd bes} in Section~\ref{section:proof of dim lemma}, we will take $H=\I(S_g)$.

\begin{cor}
\label{cor:cd bes}
Let $g \geq 2$, and let $G$ be a subgroup
of $\PMod(S_{g,p})$.  Let $F:\PMod(S_{g,p}) \to \Mod(S_g)$ be the map
induced by forgetting all of the punctures, and suppose $F(G)$ is
contained in a subgroup $H$ of $\Mod(S_g)$.  We have 
\[ \cd(G) \leq \cd(H) + (p+1). \]
\end{cor}

Since the torus has Euler characteristic zero, we will instead apply
Theorem~\ref{bes} to the once punctured torus.  We have the following
analog of Corollary~\ref{cor:cd bes} for this scenario.  In
Section~\ref{section:proof of dim lemma}, we will apply
Corollary~\ref{cor:cd bes g1} with $g=1$, $k=1$, and $H=\I(S_{1,1})=1$.

\begin{cor}
\label{cor:cd bes g1}
Let $g \geq 1$, let $p\geq k>0$, and let $F : \PMod(S_{g,p})
\to \Mod(S_{g,k})$ be the map induced by forgetting $p-k$ of the
punctures of $S_{g,p}$.  If $G < \PMod(S_{g,p})$ and $F(G) < H$, then 
\[ \cd(G) \leq \cd(H) + (p-k). \]
\end{cor}

In the genus 0 case, we only need the following.

\begin{cor}
\label{cor:cd bes g0}
For $p \geq 3$, we have $\cd(\PMod(S_{0,p})) \leq p-3$.
\end{cor}

To prove the last corollary, one needs the fact that $\PMod(S_{0,3}) = 1$.

In Section~\ref{section:kg}, we will need a specialized version of
Corollary~\ref{cor:cd bes}.  For the statement, a \emph{bounding pair
  map} is a product $T_cT_d^{-1}$, where $c$ and $d$ are disjoint,
homologous nonseparating curves.

\begin{lemma}
\label{lem:cd bes kg}
Let $g \geq 2$, and let $G$ be a subgroup
of $\PMod(S_{g,p})$.  Let $F_1:\PMod(S_{g,p}) \to \Mod(S_{g,1})$ and
$F:\PMod(S_{g,p}) \to \Mod(S_g)$ be the maps obtained by forgetting
all but one of the punctures and all of the punctures, respectively.
Suppose that $F_1(G)$ contains no power of any bounding pair maps, and
that $F(G)$ is contained in a subgroup 
$H$ of $\Mod(S_g)$.  We have 
\[ \cd(G) \leq \cd(H) + p. \]
\end{lemma}

\begin{proof}

As in the case of Corollary~\ref{cor:cd bes}, we apply
Theorem~\ref{bes} inductively.  However, to improve from $p+1$ to
$p$, we need an extra argument in the last step.

By the same straightforward applications of Fact~\ref{fact:cd ses} and Theorem~\ref{bes} used for
Corollary~\ref{cor:cd bes}, we have that
\[ \cd(G) \leq \cd(F_1(G)) + (p-1).\]
It remains to show
\[ \cd(F_1(G)) \leq \cd(H) + 1. \]
We consider the exact sequence
\[ 1 \to A \to F_1(G) \to F(G) \to 1. \]
By Theorem~\ref{bes}, the kernel $A$ is a subgroup of $\pi_1(S_g)$.
We claim that $A$ has infinite index; indeed, under the map
$\pi_1(S_{g,1}) \to \Mod(S_{g,1})$ from Theorem~\ref{bes}, an element
$\alpha$ of $\pi_1(S_{g,1})$ represented by a simple nonseparating
loop maps to a bounding pair map (cf. \cite[Figure 14]{jsb}), and so
by assumption no power of $\alpha$ is an element of $A$.  It follows
that $A$ is a free group (the corresponding cover of $S_g$ is
noncompact).  This implies that $\cd(A) \leq 1$.  It remains to apply 
Facts~\ref{fact:cd ses} and~\ref{fact:cd subgp} (subadditivity and monotonicity).
\end{proof}


\subsection{Stabilizer dimensions for \boldmath$\I(S_g)$}
\label{section:proof of dim lemma}

For the duration of this section, we fix some $g \geq 2$.
Also, throughout the section, let $M$ be a fixed multicurve in $S=S_g$
consisting entirely of nonseparating curves.  Suppose that $S-M$ has
$P$ components of positive genus $S_1, \dots, S_P$ and $Z$ components
of genus zero $S_{P+1}, \dots, S_{P+Z}$.  Say that $S_i$ is
homeomorphic to a surface of genus $g_i$ with $p_i$ punctures.

Let $G(M)$ be the free abelian group generated by the Dehn twists in
the curves of $M$.  It is a theorem of Vautaw that $G(M) \cap \I(S_g)$
is generated by bounding pair maps \cite[Theorem 3.1]{wrv}.  Let $BP$
be the number of curves of $M$ minus the number of distinct homology
classes represented by the curves of $M$.  We will use the following
consequence of Vautaw's theorem.

\begin{theorem}
\label{vautaw}
For $M$ as above, we have $G(M) \cap \I(S_g) \cong \Z^{BP}$.
\end{theorem}

\begin{lemma}
\label{lemma:stab bound}
Assume Proposition~\ref{prop:dim lemma} for all genera between 2 and $g-1$, inclusive.
In the notation of the preceding paragraphs, we have
\[ \cd(\Stab_{\I(S)}(M)) \leq \sum_{i=1}^P (3g_i+p_i-4) + \sum_{i=P+1}^{P+Z}
(p_i-3) + BP. \]
\end{lemma}

\begin{proof}

Birman--Lubotzky--McCarthy proved that we have the following short
exact sequence \cite[Lemma 2.1]{blm}:
\[ 1 \to G(M) \to \Stab_{\Mod(S)}(M) \to \Mod(S-M) \to 1. \]
Here, the group $\Stab_{\Mod(S)}(M)$ consists of those mapping classes
that fix the isotopy class of $M$, and, as above, $G(M) \cong Z^{|M|}$
is the group generated by the Dehn twists about the curves of $M$.

If we restrict to the Torelli group, work of Ivanov shows that the map
from $\Stab_{\I(S)}(M)$ to $\Mod(S-M)$ has image in $\PMod(S-M)$, and
so, if $S_i$ is a component of $S-M$, there is a map
$\Stab_{\I(S)}(M) \to \Mod(S_i)$, i.e., the components of $S-M$ are
preserved \cite[Corollary 1.8]{nvi}.

Applying Theorem~\ref{vautaw}, we arrive at the appropriate sequence
for the Torelli group:
\[ 1 \to \Z^{BP} \to \Stab_{\I(S)}(M) \to \PMod(S-M). \]

By Fact~\ref{fact:cd ses} (subadditivity), Theorem~\ref{vautaw}, and the fact that $\cd(\Z^{BP})=BP$, it remains to show that the image of
$\Stab_{\I(S)}(M)$ in $\PMod(S-M)$ has dimension at most
\[ \sum_{i=1}^P (3g_i+p_i-4) + \sum_{i=1}^Z (p_i-3). \]
Let $G_i$ be the image of $\Stab_{\I(S)}(M)$ in each $\PMod(S_i)$, and
let $G$ be the direct product of the $G_i$.  Since the image of
$\Stab_{\I(S)}(M)$ in $\PMod(S-M)$ is a subgroup of $G$, it suffices
to show that if $g_i=0$, then $\cd(G_i) \leq p_i-3$, and
if $g_i > 0$, then $\cd(G_i) \leq 3g_i+p_i-4$ (apply
Facts~\ref{fact:cd subgp} and~\ref{fact:cd ses} (monotonicity and subadditivity)).

The genus zero case follows immediately from Fact~\ref{fact:cd subgp} (monotonicity) and
the inequality $\cd(\PMod(S_{0,p})) \leq p-3$ (Corollary~\ref{cor:cd
  bes g0}).

Now suppose $S_i$ has genus 1.  Let $\hat S_i$ be the once-punctured
torus obtained from $S_i$ by forgetting all of the punctures except
for one, and let $\Phi : S_i \to \hat S_i$ and $G_i \to \Mod(\hat
S_i)$ be the induced maps.  We have the following key observation.

\noindent \emph{Observation:} \label{obs} The image of $G_i$ is a subgroup of $\I(\hat S_i)$.

The point is that, given any curve $c$ in $S_i$, and any $f \in
\Stab_{\I(S)}(M)$, we have that $f(c)$ lies in $S_i$, and if $\Sigma$
is the surface that proves $c$ and $f(c)$ homologous, then
$\Phi(\Sigma)$ shows that $\Phi(c)$ and $\Phi(f(c))$ are homologous.

Now, we know that $\I(S_{1,1})=1$, and so the observation combined with
Corollary~\ref{cor:cd bes g1} gives $\cd(G_i) \leq 0 + (p_i-1)$.  The
latter happens to be equal to $3g_i+p_i-4$.

The case $g_i \geq 2$ is similar.  Here, we take $\hat S_i$ to be the
closed surface obtained from $S_i$ by forgetting all of the punctures.
The observation that $G_i$ is a subgroup of $\I(\hat S_i)$ is still
valid for the same reason as above.  Since $g_i < g$, we can apply Proposition~\ref{prop:dim
  lemma}, which tells us that $\cd(\I(\hat S_i)) = 3g_i-5$.  Applying the
observation and Corollary~\ref{cor:cd bes}, we arrive at $\cd(G_i)
\leq (3g_i-5)+(p_i+1) = 3g_i+p_i-4$.  This completes the proof.
\end{proof}

\begin{figure}[htb]
\psfrag{a1}{$a_1$}
\psfrag{a2}{$a_2$}
\psfrag{a3}{$a_3$}
\psfrag{b1}{$b_1$}
\psfrag{b2}{$b_2$}
\psfrag{c1}{$c_1$}
\psfrag{c2}{$c_2$}
\psfrag{d}{$d$}
\centerline{\includegraphics[scale=.5]{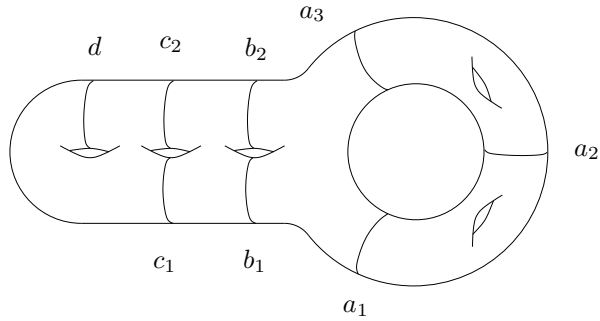}}
\caption{A multicurve in $S_6$ with $BP=4$, $D=4$, and $P=2$.}
\label{bppic}
\end{figure}

We now give another version of Lemma~\ref{lemma:stab bound}, obtained by
algebraic manipulation of the upper bound.

\begin{lemma}
\label{lemma:stab bound 2}
With the same assumptions and notation as Lemma~\ref{lemma:stab
  bound}, we have
\[ \cd(\Stab_{\I(S)}(M)) \leq 3g-3-P-|M| + BP . \]
\end{lemma}

Since the quantity $BP$ is equal to $|M|-C$, where $C$ is the number
of distinct homology classes represented by the curves of $M$, the
upper bound of the lemma can be replaced with the equivalent one,
$3g-3-P-C$.

\begin{proof}

By Lemma~\ref{lemma:stab bound}, it suffices to show that
\[ \sum_{i=1}^P (3g_i+p_i-4) + \sum_{i=P+1}^{P+Z}(p_i-3) + BP = 3g-3-P-|M| + BP. \]
We have
\begin{eqnarray*}
\sum_{i=1}^P (3g_i+p_i-4) + \sum_{i=P+1}^{P+Z}(p_i-3) + BP &=& 3\sum g_i + 2|M| -4P - 3Z + BP \\
&=& 3(g-D)+2|M|-4P-3Z+BP \\
&=& 3g-3(|M|+1)+2|M|-P+BP \\
&=& 3g-3-P-|M| + BP. \\
\end{eqnarray*}
For the first equality, we used the fact that $\sum p_i = 2|M|$.
Then, we applied the fact that $D=g-\sum g_i$.  The third equality is
a consequence of Lemma~\ref{lemma:cell dimension}, which says that
$D+P+Z = |M|+1$, and the last equality is obtained by simple algebra.
\end{proof}

The next lemma is a basic fact about multicurves made up of nonseparating curves.

\begin{lemma}
\label{lemma:bpdp}
For the multicurve $M$, we have $BP + 2 \leq D + P$.
\end{lemma}

\begin{proof}

We start with an observation.  For any multicurve in $S_g$ with $g
\geq 2$, it is straightforward to see that $D+P \geq 2$ (recall that multicurves are nonempty).

We now proceed by induction on genus, with the base case $g=2$.  In that
case, $BP = 0$ (there are no bounding pairs in $S_2$).  Since $D+P
\geq 2$, we are done in this case.

Now assume $g \geq 3$, and that the lemma is true for all genera
between 2 and $g-1$, inclusive.  Let $M$ be a multicurve in
$S=S_g$.  If there are no bounding pairs, then, as above, there is
nothing to do.  So assume $BP>0$.

We can find an ``innermost bounding pair'' of $M$, by which we mean a
bounding pair $\{c,d\} \subseteq M$ so that one component of $S-(c
\cup d)$ contains no other bounding pairs of $M$.  For example, in
Figure~\ref{bppic}, the bounding pairs $\{a_1,a_2\}$, $\{a_2, a_3\}$, and
$\{c_1,c_2\}$ are innermost, whereas the bounding pairs $\{a_1,a_3\}$
and $\{b_1,b_2\}$ are not innermost.

When there is at least one bounding pair, then there is an innermost one: if
there is a bounding pair on one side of a given bounding pair, then
then the new bounding pair bounds a strictly smaller subsurface (as
measured by Euler characteristic), and so the process of finding
increasingly innermore bounding pairs must terminate.

Let $\{c,d\}$ be an innermost bounding pair of $M$, and let $S_1$ be a
component of $S-(c \cup d)$ that contains no bounding pairs of $M$.
We obtain a new surface $S'$ from $S$ by replacing $S_1$ with an
annulus.  There is an induced multicurve $M'$ on $S'$, after
removing the image of either $c$ or $d$ (they become homotopic by
construction).

We claim that the genus of $S'$ lies in $[2,g-1]$.  Indeed, the
closures of $S_1$ and $S-S_1$ have two boundary components each, and
so if either had genus zero, we would see that $c$ and $d$ were
homotopic.

By induction, $M'$ satisfies
\[ BP' + 2 \leq D' + P' \]
where $BP'$, $D'$, and $P'$ are defined are defined for
$M'$ in the same way as
$BP$, $D$, and $P$ for $M$. 
Since $\{c,d\}$ was innermost in $S$, we have
\[ BP' = BP-1. \]
Also we have
\[ D' + P' + 1 \leq D + P \]
because when we replace $S_1$ with the annulus, we are either deleting
a positive genus component of $S-M$, or we see that there is at least
one nonseparating curve in $S_1$, which is automatically homologically
independent of the other curves of $M$.  In the
latter case, replacing $S_1$ by an annulus reduces $D$.
Thus, we have
\[ BP+2 = BP' + 3 \leq D'+P' + 1 \leq D+P \]
and the proof is complete.
\end{proof}

We now prove Proposition~\ref{prop:dim lemma}, which completes the
proof of Theorem~\ref{main:torelli}.

\begin{proof}[Proof of Proposition~\ref{prop:dim lemma}]

We proceed by strong induction.  Assume that Proposition~\ref{prop:dim lemma}
holds for genera between 2 and $g-1$ inclusive (if $g=2$ this is an
empty assumption).  We want to prove the following inequality:
\[ \cd (\Stab_{\I(S_g)}(M)) + B(M) \leq 3g-5. \]
We can replace $B(M)$ with $|M|-D$ (Lemma~\ref{lemma:cell dimension}).
Applying Lemma~\ref{lemma:stab bound 2}, it is enough to show that
\[ (3g-3-P-|M|+BP) + (|M|-D) \leq 3g-5.  \]
But this inequality is equivalent to the inequality $BP+2 \leq D+P$,
and so an application of Lemma~\ref{lemma:bpdp} completes the proof.
\end{proof}


\subsection{Stabilizer dimensions for \boldmath$\K(S_g)$}
\label{section:kg}

Throughout this section we again fix some $g \geq 2$ and some multicurve $M$
in $S_g$ consisting entirely of nonseparating curves.

We will require the analogue of Theorem~\ref{vautaw} for
$\K(S_g)$, recently proven by the authors \cite{multik}.

\begin{theorem}
\label{thm:multitwist kg}
For $M$ as above, we have $G(M) \cap \K(S_g) = 1$.
\end{theorem}

For once punctured surfaces, we will only need the following much
weaker statement, proven by Johnson \cite[Lemmas 4A and 4B]{djabelian}.

\begin{lemma}
\label{lemma:no bps}
The group $\K(S_{g,1})$ contains no nontrivial powers of any bounding
pair maps.
\end{lemma}

As before, let $S_{g,p}$ denote a surface of genus $g$ with $p$
punctures.  For any surface $S=S_{g,p}$, we denote by $\K(S)$ the
subgroup of $\Mod(S)$ generated by Dehn twists about separating
curves.

\begin{lemma}
\label{lemma:kg bes}
Let $G < \K(S_{g,p})$ with $g \geq 2$.  We have
\[ \cd(G) \leq \cd(\K(S_g)) + p. \]  
\end{lemma}

\begin{proof}

It suffices to show that Lemma~\ref{lem:cd bes kg} applies.  First,
since the generators for $\K(S_{g,p})$ do not permute punctures, we
have that $G < \PMod(S_{g,p})$.  Also, since separating curves remain
separating when punctures are forgotten, we see that the images of $G$
in $\Mod(S_{g,1})$ and $\Mod(S_g)$ are subgroups of $\K(S_{g,1})$ and
$\K(S_g)$, respectively.  By Lemma~\ref{lemma:no bps}, $\K(S_{g,1})$
contains no bounding pair maps.  Thus, Lemma~\ref{lem:cd bes kg}
applies and immediately gives the desired statement.
\end{proof}

We recall the notation from the start of Section~\ref{section:proof of
  dim lemma}.  We label the positive genus components of $S-M$ by
$S_1, \dots, S_P$, and the genus zero components by
$S_{P+1},\dots,S_{P+Z}$, and we say that $S_i$ is a surface of genus
$g_i$ with $p_i$ punctures.  We have the following analogue of
Lemma~\ref{lemma:stab bound}.

\begin{lemma}
\label{lemma:stab bound kg}
Assume Proposition~\ref{prop:dim lemma kg} for all genera between 2 and $g-1$, inclusive.
In the notation of the preceding paragraph, we have
\[ \cd(\Stab_{\K(S_g)}(M)) \leq \sum_{i=1}^P (2g_i+p_i-3) + \sum_{i=P+1}^{P+Z}
(p_i-3). \]
\end{lemma}

\begin{proof}

The proof is essentially the same as the proof of
Lemma~\ref{lemma:stab bound}, with lemmas about $\I(S_g)$ replaced by
lemmas about $\K(S_g)$.

As in the proof of Lemma~\ref{lemma:stab bound}, we cut $S$ along $M$
and obtain
\[ 1 \to G(M) \cap \K(S) \to \Stab_{\K(S)}(M) \to \PMod(S-M) \]
where $G(M)$ is the group generated by Dehn twists in the curves of
$M$.  By Theorem~\ref{thm:multitwist kg}, the term $G(M) \cap \K(S)$ is
trivial, and so $\Stab_{\K(S)}(M)$ is isomorphic to its image in
$\PMod(S-M)$.  So we must show that this image has cohomological
dimension bounded above by
\[ \sum_{i=1}^P (2g_i+p_i-3) + \sum_{i=1}^Z (p_i-3). \]
As in the proof of Lemma~\ref{lemma:stab bound}, we denote by $G_i$
the image of $\Stab_{\K(S)}(M)$ in $\PMod(S_i)$.

Like before, it suffices to show that if $g_i=0$, then $\cd(G_i) \leq
p_i-3$, and if $g_i > 0$, then $\cd(G_i) \leq 2g_i+p_i-3$.

As in Lemma~\ref{lemma:stab bound}, the genus zero case follows from
Corollary~\ref{cor:cd bes g0}.  For $g_i=1$, note that $2g_i+p_i-3 =
3g_i+p_i-4 = p_i-1$, and so, by Fact~\ref{fact:cd subgp} (monotonicity), the upper
bound of $p_i-1$ obtained from Corollary~\ref{cor:cd bes g1} in the
proof of Proposition~\ref{prop:dim lemma kg} is sufficient.

It remains to deal with the case $g_i \geq 2$.  Let $\hat
S_i$ be the closed surface obtained from $S_i$ by forgetting all of
the punctures.  Since separating curves in $S$ remain separating in
$S_i$, we have $G_i < \K(S_i)$, and so Lemma~\ref{lemma:kg bes} gives
that $\cd(G_i) \leq \cd(\K(\hat S_i)) + p_i$.  By induction,
$\cd(\K(\hat S_i)) \leq 2g_i-3$, and this completes the proof.
\end{proof}

We can simplify the upper bound of Lemma~\ref{lemma:stab bound kg} as
follows.  The proof is analogous to the argument for Lemma~\ref{lemma:stab bound 2}, and is left to the reader.

\begin{lemma}
\label{lemma:stab bound kg 2}
With the assumptions and notation of Lemma~\ref{lemma:stab bound kg},
we have
\[ \cd (\Stab_{\K(S_g)}(M)) \leq 2g-3+D-|M|. \]
\end{lemma}

We are now ready to prove the upper bound for Theorem~\ref{main:kg}.

\begin{proof}[Proof of Proposition~\ref{prop:dim lemma kg}]

We proceed as in the proof of Proposition~\ref{prop:dim lemma}.  Assume by
induction that Proposition~\ref{prop:dim lemma kg} is true for all genera between
2 and $g-1$, inclusive.  The goal is to show the following:
\[ \cd (\Stab_{\K(S_g)}(M)) + B(M) \leq 2g-3. \]
By Lemma~\ref{lemma:stab bound kg 2} and Lemma~\ref{lemma:cell dimension} it suffices
to show
\[ (2g-3+D-|M|) + (|M|-D) \leq 2g-3.  \]
But this inequality simplifies to $0 \leq 0$.
\end{proof}

\p{Remark.} We point out that in the case of $\K(S_g)$ we did not need
an analog of Lemma~\ref{lemma:bpdp}.  In this sense, the upper bound
for Theorem~\ref{main:kg} is slightly easier than that for
Theorem~\ref{main:torelli}.  Also recall that the lower bound was also
easier for Theorem~\ref{main:kg}.


\section{The Mess description of $\I(S_2)$}
\label{section:genus 2}

We now turn to the proof of Theorem~\ref{main:mess}.  As
discussed in the introduction, Theorem~\ref{main:torelli}
implies that $\I(S_2)$ is a free group.  However, the statement of Theorem~\ref{main:mess} is much stronger, and the infinite generation is needed for the base case of Theorem~\ref{main:top}.

The proof breaks up into three parts: we first analyze the stabilizer
in $\I(S_2)$ of a pair of nonseparating curves, then we analyze the
stabilizer of a single nonseparating curve, and finally we piece these
together to get the desired description of $\I(S_2)$.  We can think of
these steps as the genus 0, genus 1, and genus 2 versions of the Mess theorem.

\p{Genus 0.} We start by investigating the stabilizer in $\I(S_2)$ of
a vertex of $\B(S_2)$ corresponding to a pair of disjoint
nonseparating curves.  Since the complement in $S_2$ has genus zero,
we think of this as a genus zero version of Theorem~\ref{main:mess}.

In the statement of the lemma, we say that a symplectic splitting of
$H_1(S_2,\Z)$ is \emph{compatible} with a pair of nonseparating curves
$\{a,b\}$ (or the associated homology classes) if the homology class
$[a]$ lies in one subspace determined by the splitting, and
$[b]$ lies in the other.

\begin{lemma}
\label{lem:2 pairs}
Let $M$ be the multicurve consisting of two nonseparating curves $a$
and $b$ in $S_2$.  The group $\Stab_{\I(S_2)}(M)$ is an infinitely
generated free group, with one Dehn twist generator for each
symplectic splitting compatible with $M$.
\end{lemma}

\begin{proof}

We cut $S_2$ along $a$ and $b$ and obtain a sphere $S'$ with 4
punctures.  Let $X$ be the flag complex with a vertex for each isotopy
class of curves in $S'$ that comes from a separating curve in $S_2$,
and an edge for each pair of isotopy classes with geometric
intersection number 4.

The complex $X$ is contractible: if we think of the vertices as
isotopy classes of arcs in $S'$ connecting the two punctures coming from $a$,
then edges correspond to disjointness of arcs, and the required
statement is given by Harer \cite[Theorem 1.6]{jh} (the argument of
\cite{ah3} also applies, with minor modifications).  We also see from
this point of view that $X$ is a graph (so it is a tree)---any three arcs in $S'$ that represent distinct separating curves must intersect.

The group $\Stab_{\I(S_2)}(M)$ acts on $X$.  The quotient has one
vertex for each homology splitting of $H_1(S_2,\Z)$ compatible with $M$;
in other words, if two separating curves $c$ and $c'$ are disjoint
from $M$ and induce the same homology splitting, then there is an
element of $\I(S_2)$ that fixes $a$ and $b$ and takes $c$ to $c'$.
To construct such a mapping class, cut one copy of $S$ along $c$ and
another copy of $S$ along $c'$, choose the unique maps on the
punctured tori that act trivially on homology, and then choose any
element of $\Mod(S_2)$ that restricts to the chosen maps on the
punctured tori (cf. \cite[Theorem 1A]{dj} and \cite[Lemma A.3]{ap}).

Our goal is to show that the quotient complex is contractible (a
tree).  Since vertex stabilizers are the Dehn twists about the
corresponding separating curves in $S_2$, and edge stabilizers are
trivial (Theorem~\ref{vautaw}, plus the fact that any two nonisotopic
separating curves in $S_2$ must intersect), the lemma will then follow
from the classical Bass--Serre theory of group actions on trees
\cite[Theorem 13]{trees}.

Let $([a],[a'],[b],[b'])$ be a symplectic basis for $H_1(S_2,\Z)$,
where $a$ and $b$ are the curves in the statement of the lemma (take
the algebraic intersection of $[a]$ and $[a']$, say, to be 1).  It
follows from elementary linear algebra that the splittings of
$H_1(S_2,\Z)$ compatible with $M$ are in bijection with the cosets
$([a']+k[b])+\langle[a]\rangle$, and hence are indexed by $\Z$.

We now claim that two vertices form an edge in the quotient complex if
and only if they are adjacent in $\Z$.  It will follow that the
quotient complex is isomorphic to the standard triangulation of $\R$.

If two vertices are adjacent in $\Z$, then one can draw disjoint arcs
in $S'$ corresponding to those vertices.  This is proven by drawing
the picture: if $\alpha_0$ is an arc in $S'$ corresponding to the
integer $0$, then $T_c^k(\alpha_0)$ is an arc corresponding to $k$,
where $T_c$ is the Dehn twist about any fixed curve in $S'$ that
comes from a curve in $S_2$, represents the class $[a]+[b]$, and
intersects $\alpha_0$ in one point; see Figure~\ref{akpic} for an
illustration.  One can check that $T_c^k(\alpha_0)$ and
$T_c^{k+1}(\alpha_0)$ are disjoint, and so the vertices corresponding
to $k$ and $k+1$ are connected by an edge in the quotient.

If two vertices labelled by $k,k' \in \Z$ are connected in the
quotient, this means that there are corresponding arcs in $S'$ that
are disjoint, and hence corresponding curves in $S_2$ that are
disjoint.  As above, these curves represent primitive homology classes
$[a']+k[b]+j[a]$ and $[a']+k'[b]+j'[a]$.  Since the curves are
disjoint, we have $j=j'$ (otherwise, the algebraic intersection is
nonzero).  The curves in $S_2$ span a maximal isotropic subspace
(i.e. a maximal subspace where the intersection form is trivial); in
particular, they span the subspace spanned by $[a']$ and $[b]$.
Hence, the determinant of the pair of vectors $[a']+k[b]$ and
$[a']+k'[b]$ must be $\pm 1$, and we see $k=k'\pm 1$, which is what we
wanted to show.
\end{proof}

\begin{figure}[htb]
\psfrag{a0}{$\alpha_0$}
\psfrag{a2}{$\alpha_2$}
\psfrag{a3}{$\alpha_3$}
\psfrag{c}{$c$}
\centerline{\includegraphics[scale=.5]{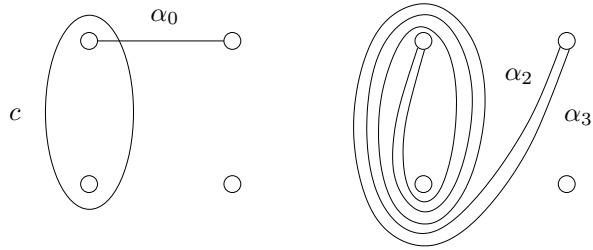}}
\caption{The arcs $\alpha_0$, $\alpha_2$, and $\alpha_3$ from the
  proof of Lemma~\ref{lem:2 pairs}.}
\label{akpic}
\end{figure}

\p{Genus 1.} We now need to understand the stabilizer in $\I(S_2)$ of
a vertex of $\B(S_2)$ corresponding to a single nonseparating curve
$c$ in the class $x$.  This is the genus 1 version of
Theorem~\ref{main:mess}.  We say that a homology splitting is \emph{compatible with $x$} if $x$ lies in one component of the splitting.

\begin{lemma}
\label{lem:g1 pair}
The group $\Stab_{\I(S_2)}(c)$ is an infinitely generated free group,
with one Dehn twist generator for each symplectic splitting compatible
with $x$.
\end{lemma}

\begin{proof}

We cut $S_2$ along $c$ and obtain a twice-punctured torus $S'$.  To
find a generating set for $\Stab_{\I(S_2)}(c)$, we consider the action
on the complex of basic cycles for $S'$.  To this end, we fix a
homology class $y \in H_1(S_2,\Z)$ coming from a nonseparating curve
$d$ in $S_2$ that is disjoint (and isotopically distinct) from $c$.  The classes $x$ and $y$
correspond to classes in $H_1(S',\Z)$, which we also call $x$ and
$y$.

The choice of the class $y$ gives rise to a complex of
cycles $\B(S')$, which is contractible (the argument is the same
as the closed case, with only cosmetic changes).

Given any two disjoint (isotopically distinct) nonseparating curves in
$S'$, their homology classes necessarily differ by $\pm x$.  It
follows that an integral basic cycle for $y$ consists of a single
curve in the homology class $y+kx$ (along with $-kc$) and that
$\B(S')$ is a tree.  We also deduce that the vertices of the quotient
are indexed by $\Z$, and that if two vertices in the quotient are
connected by an edge, then they are adjacent in $\Z$.  Also, if two
vertices are adjacent in $\Z$, then we can realize them by disjoint
curves.  Thus the quotient of $\B(S')$ by $\Stab_{\I(S_2)}(c)$ is
isomorphic to the standard triangulation of $\R$.

As in Lemma~\ref{lem:2 pairs}, $\Stab_{\I(S_2)}(c)$ is freely
generated by the vertex stabilizers (we take one copy for each vertex
of the quotient).  By Lemma~\ref{lem:2 pairs}, each vertex stabilizer
is an infinitely generated free group, with one Dehn twist generator
for each homology splitting compatible with $x$ and $y+kx$.

It remains to establish the bijection between generators of
$\Stab_{\I(S_2)}(c)$ and homology splittings compatible with $x$.  By
Lemma~\ref{lem:2 pairs}, a generator corresponds to a unique homology
splitting compatible with $x$.  On the other hand, given a homology
splitting compatible with $x$, we consider the induced splitting of
$y$, and we get a unique generator.  It is clear that these
identifications are inverses of each other, so we are done.
\end{proof}

\p{Genus 2.} We finally use Lemmas~\ref{lem:2 pairs} and~\ref{lem:g1
  pair} to obtain the full version of the Mess theorem.

\begin{proof}[Proof of Theorem~\ref{main:mess}]

By Theorem~\ref{main:b}, the complex of cycles $\B(S_2)$ is contractible.  We also have that it is one dimensional
by Lemma~\ref{lemma:cell dimension}, and so it is in fact a tree.  As in the previous
lemmas, we need to argue that the quotient $\B(S_2)/\I(S_2)$ is
contractible, and then do some bookkeeping to check that there is
exactly one generator for each homology splitting.

The vertices of $\B(S_2)/\I(S_2)$ are in one to one correspondence
with the homology classes of basic cycles for $x$; that is, if two
nonseparating curves, or two pairs of nonseparating curves, in $S_2$
represent the same homology classes, then there is an element of
$\I(S_2)$ taking one to the other \cite[Lemma A.3]{ap}.  There is a
single distinguished vertex corresponding to the homology class $x$,
and infinitely many vertices corresponding to pairs of homology
classes (a basic cycle in $S_2$ has one or two curves).

In order to show that $\B(S_2)/\I(S_2)$ is contractible, we will
assign a ``weight'' to each vertex, and show that each vertex is
connected by an edge to exactly one vertex of smaller weight.  We
declare the weight of the distinguished vertex corresponding to a single curve in
the class $x$ to be 1.  For a vertex given by a basic cycle
$pa+qb$ with $p,q > 0$, we define the weight to be $p+q$.

Let $v$ be the vertex corresponding to the positive cycle $pa+qb$.
There are exactly two classes that can be represented disjointly from
the multicurve $\{a,b\}$, namely, $[a]+[b]$ and $[a]-[b]$.  Hence,
there are four potentially adjacent vertices, given by the pairings of
$[a]$ and $[b]$ with these two new classes.

The class $[a]-[b]$ cannot pair with either
$[a]$ or $[b]$ to give a vertex of smaller weight.  Indeed, we have
$(p+q)[a]+q([b]-[a]) = x$ and $(p+q)[b]+p([a]-[b])=x$, and so in
either case the weight is larger than $p+q$.

The weights of the other two nearby vertices are both smaller than
$p+q$, since we have $(q-p)[b]+p([a]+[b])=x$ and
$(p-q)[a]+q([a]+[b])=x$.  As $q-p$ is negative, it follows from
Lemma~\ref{lemma:orient} that $v$ is not connected by an edge to the
vertex corresponding to the first equation.  Since $p-q$ is positive,
it follows from Lemma~\ref{lemma:covariant} that $v$ is connected to
the other vertex.  This completes the proof of contractibility of
$\B(S_2)/\I(S_2)$.

Now for the bookkeeping.  The stabilizer of any lift of the distinguished
vertex is an infinitely generated free group generated by one Dehn
twist generator for each homology splitting compatible with $x$ by
Lemma~\ref{lem:g1 pair}.  The stabilizer of a lift of a vertex
corresponding to two homology classes is an infinitely generated free
group generated by one Dehn twist for each splitting compatible with
that pair by Lemma~\ref{lem:2 pairs}.  The theorem follows.
\end{proof}

\p{Remark.} It is more illuminating to draw a diagram of
$\B(S_2)/\I(S_2)$.  It is naturally subdivided into pieces,
corresponding to different 2-dimensional isotropic subspaces of
$H_1(S,\Z)$ that contain $x$.  Each such subspace is group isomorphic
to $\Z^2$.  The set of bases for $\Z^2$ are depicted via the Farey
graph: vertices correspond to (unsigned) primitive vectors, and edges
correspond to bases.  There is a distinguished vertex $x$, and the
rest of the component of $\B(S_2)/\I(S_2)$ corresponding to this
subspace is the tree shown in Figure~\ref{g2fareypic}.  The entire
complex $\B(S_2)/\I(S_2)$ is obtained by gluing infinitely many of
these trees along their distinguished vertices.

\begin{figure}[htb]
\centerline{\includegraphics[scale=.5]{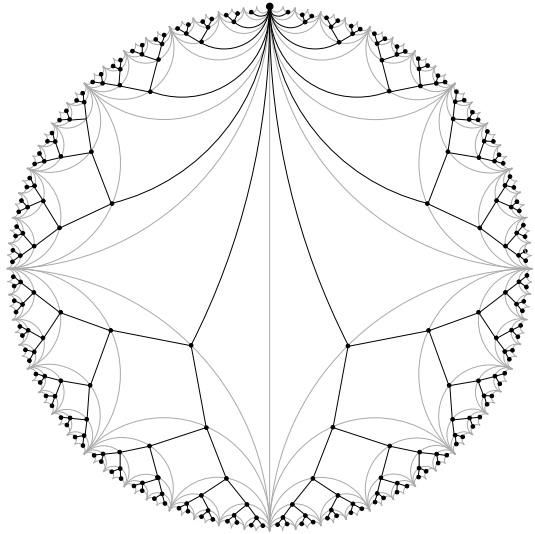}}
\caption{A piece of $\B(S_2)/\I(S_2)$ corresponding to a particular
  Lagrangian subspace of $H_1(S_2,\Z)$.  The Farey graph is drawn
  lightly for reference.}
\label{g2fareypic}
\end{figure}


\section{Infinite generation of homology}
\label{section:top}

The goal of this section is to prove Theorem~\ref{main:top}, which
states that, for $g \geq 2$, the group $H_{3g-5}(\I(S_g),\Z)$ is
infinitely generated.  The basic idea is to use induction (the base
case is Theorem~\ref{main:mess}) and the spectral sequences
associated to the appropriate Birman exact sequences.

Throughout this section, we will be forced to consider homology with various local and global coefficients.  When the coefficient ring is $\Z$, it will be convenient to omit this from the notation.

\subsection{Spectral sequences}
\label{section:spectral}

We recall some facts about spectral sequences;
for an introduction to the theory, see \cite{ah2}, \cite{ffg}, or \cite{ksb}.

\p{Hochschild--Serre spectral sequence.} Given a short exact sequence of groups
\[ 1 \to K \to G \to Q \to 1 \]
there is an associated Hochschild--Serre spectral sequence.  The second
page of this spectral sequence is the two-dimensional array of
homology groups
\begin{equation*}
E^2_{p,q} = \begin{cases}
H_p(Q,H_q(K)) & p,q \geq 0 \\
0 & \mbox{otherwise}
\end{cases}
\end{equation*}
where the coefficients $H_q(K)$ are considered to be local: $Q$ acts on
$K$ by conjugation.  The main difficulties in calculation come from this action.

We remark that Fact~\ref{fact:cd ses} is an easy consequence of the Hochschild--Serre spectral sequence.

\begin{fact}
\label{fact:spectral}
Consider the Hochschild--Serre spectral sequence associated to a short exact sequence
\[ 1 \to K \to G \to Q \to 1 \]
and assume $\cd(Q)$ and $\cd(K)$ are finite.  For
$p=\cd(Q)$ and $q=\cd(K)$, we have
\[ H_{p+q}(G) \cong E^2_{p,q} = H_p(Q,H_q(K)). \]
\end{fact}

\p{Cartan--Leray spectral sequence.} For a group $G$ acting cellularly
and without rotations on a contractible complex $X$, we can associate
the Cartan--Leray spectral sequence (to act \emph{without rotations} means
that, if a cell is fixed by some element, then the cell is fixed
pointwise).  The first page of this spectral sequence is given by
\begin{equation*}
E^1_{p,q} = \begin{cases}
\displaystyle\bigoplus_{\sigma \in X_p} H_q(G_\sigma) & p,q \geq 0 \\
0 & \mbox{otherwise}
\end{cases}
\end{equation*}
where $X_p$ is the set of $p$-cells of $X/G$, and $G_{\sigma}$ is the
stabilizer in $G$ of a lift of the cell $\sigma$.

We only need the following basic fact; see \cite[Section
  VII.7]{ksb} and \cite[Section 18]{ffg}.

\begin{fact}
\label{fact:cartan}
Consider the Cartan--Leray spectral sequence associated to the action
of the group $G$ on the contractible space $X$.  Suppose that
$E^1_{p,q}=0$ for $p+q > k$.  Then $E^1_{0,k} = \oplus H_k(G_v)$
injects as a subgroup of $H_k(G)$ (the sum is over all vertices $v$ of
$X/G$).
\end{fact}

We remark that Quillen's upper bound on the cohomological dimension of
a group given in the introduction is an immediate consequence of the
Cartan--Leray spectral sequence; this is Exercise 4 in Section VIII.2
of \cite{ksb}.


\subsection{Torelli Birman exact sequences}

In this section, we give two versions of the Birman exact sequence
that are particular to the Torelli group: the first, due to Putman,
allows us to pass from two punctures to one puncture; the second, due
to Johnson, enables us to go from there to the closed case.  For
each sequence, we analyze the action of the quotient on the kernel.

Let $v$ be a vertex of $\B(S_g)$ corresponding to a single curve $c$
(the image of $v$ in $\B(S_g)/\I(S_g)$ is the distinguished vertex).
We can think of the stabilizer in $\I(S_g)$ of $v$ (equivalently, of
$c$) as a subgroup of $\Mod(S_{g-1,2})$, and so, as in the Birman
exact sequence (Theorem~\ref{bes}), there is a forgetful map
\[ \Stab_{\I(S_g)}(v) \to \Mod(S_{g-1,1}) \]
whose kernel is a subgroup of $\pi_1(S_{g-1,1})$.

\begin{lemma}
\label{seq1}
The image of $\Stab_{\I(S_g)}(v)$ in $\Mod(S_{g-1,1})$ is a
subgroup of $\I(S_{g-1,1})$.  Moreover, the sequence
\[
1 \to K \to \Stab_{\I(S_g)}(v) \to \I(S_{g-1,1}) \to 1
\]
where $K$ is the commutator subgroup of $\pi_1(S_{g-1,1})$, is exact.
\end{lemma}

\begin{proof}

The first statement of the lemma is proven in the same way as the
observation in the proof of Lemma~\ref{lemma:stab bound} (page \pageref{obs}).  The
surjectivity of the map $\Stab_{\I(S_g)}(v) \to \I(S_{g-1,1})$ can be
seen as follows: embed $S_{g-1,1}$ in $S_{g-1,2}$ so that
$S_{g-1,2}-S_{g-1,1}$ is a twice-punctured disk; given an element of
$\I(S_{g-1,1})$ we extend it by the identity to get an element of
$\Stab_{\I(S_g)}(v)$.  The description of $K$ is due to Putman
\cite[Theorem 4.1]{ap} (Putman uses surfaces with boundary instead of
with punctures, but only inconsequential changes to his proof are
required).
\end{proof}

We also have the following short exact sequence, due to Johnson
\cite{dj1}.

\begin{lemma}
\label{seq2}
For $g \geq 2$, we have
\[ 1 \to \pi_1(S_{g}) \to \I(S_{g,1}) \to \I(S_{g}) \to 1. \]
\end{lemma}

As indicated in Section~\ref{section:spectral}, we can apply
the Hochschild--Serre spectral sequence to understand the groups in
Lemmas~\ref{seq1} and~\ref{seq2}.  Since coefficients are local, we will 
 need to understand the actions of the quotient groups on
the kernels of the two sequences.

In the next lemma, we consider the action (by conjugation) of
$\I(S_{g})$ on $H_i(\pi_1(S_{g})) \cong H_i(S_{g})$ coming from the
sequence in Lemma~\ref{seq2}.  This is the restriction of the usual
action of $\Mod(S_g)$ on $H_1(S_g)$.

\begin{lemma}
\label{lem:trivial action}
The action of $\I(S_{g})$ on $H_i(S_g)$ is trivial for $i=0,1,2$.
\end{lemma}

\begin{proof}

The case $i=0$ is trivial.  When $i=1$, the statement of the lemma is
the definition of $\I(S_{g})$.  For $i=2$, the lemma is equivalent
to the fact that elements of $\I(S_{g})$ preserve the orientation of $S_{g}$.
\end{proof}

We now consider the action of $\I(S_{g-1,1})$ on $H_1(K)$ coming from
Lemma~\ref{seq1}.  Unlike the previous case, the action here is
nontrivial (for instance, take an element of $K$ that is the
commutator of two simple loops that intersect only at the basepoint,
and consider the action by a Dehn twist about a separating curve that
is disjoint from the first loop and intersects the other one twice).
However, we will be able to find a subgroup of $H_1(K)$ on which
$\I(S_{g-1,1})$ acts trivially (Lemma~\ref{lem:global module} below).
In the proof of Theorem~\ref{main:top}, this will allow us to use
global coefficents.

First, we have a technical lemma about commutators.

\begin{lemma}
\label{lem:commutator}
Let $F=\langle x_1,\cdots,x_{2n} \rangle$ be the free group generated by
the $x_i$, let $K$ be the commutator subgroup of $F$, and let $h \in
K$ be the product of commutators $[x_1,x_2]\cdots[x_{2n-1},x_{2n}]$.  
The class of $h$ in $H_1(K)$ is nontrivial.
\end{lemma}

\begin{proof}

The group $H_1(K)$ can be thought of as the first homology of the
1-skeleton of the standard cubing of $\R^{2n}$.  It is clear now that
$h$ represents a nontrivial class.
\end{proof}

As above, we consider the action by conjugation of $\I(S_{g-1,1})$ on
$H_1(K)$ coming from Lemma~\ref{seq1}.  We now give an invariant
submodule of $H_1(K)$.

\begin{lemma}
\label{lem:global module}
There is a nontrivial subgroup of $H_1(K)$ on which $\I(S_{g-1,1})$
acts trivially.
\end{lemma}

\begin{proof}

Let $\alpha_1, \beta_1, \dots, \alpha_{g-1}, \beta_{g-1}$ be the usual
generators for $\pi_1(S_{g-1,1})$, and let $\delta$ be the following
element of $\pi_1(S_{g-1,1})$:
\[ \delta = [\alpha_1,\beta_1]\cdots[\alpha_{g-1},\beta_{g-1}]. \]
Clearly, $\delta \in K$.  We will show that the class of $\delta$ in
$H_1(K)$ is fixed by $\I(S_{g-1,1})$.  By Lemma~\ref{lem:commutator},
this is a nontrivial element of $H_1(K)$, and so the module generated by
this element is the desired submodule of $H_1(K)$.

Thinking of $\Stab_{\I(S_g)}(v)$ as a subgroup of $\Mod(S_{g-1,2})$,
the image of $\delta$ under the map $K \to \Stab_{\I(S_g)}(v)$ in
Lemma~\ref{seq1} is the Dehn twist $T_d$ about a curve $d$ that bounds a
twice-punctured disk in $S_{g-1,2}$ (we think of $\delta$ as the
simple loop in this disk that is based at one puncture and surrounds
the other).

To prove the lemma, it suffices to show that, given an arbitrary
element $f$ of $\I(S_{g-1,1})$, there is a lift $F$ of $f$ to
$\Stab_{\I(S_g)}(v)$ (thought of as a subgroup of $\Mod(S_{g-1,2})$)
so that $FT_dF^{-1}=T_{F(d)}$ is conjugate in $K$ to $T_d$.  In fact,
we will be able to construct the lift $F$ so that $F(d)=d$, and so
$T_{F(d)}$ is actually equal to $T_d$.

Let $S_{g-1,1} \to S_{g-1,2}$ be a fixed embedding with the property
that $S_{g-1,2}-S_{g-1,1}$ is the twice-punctured disk bounded by $d$.
Via this embedding, any homeomorphism of $S_{g-1,1}$ can be regarded
as a homeomorphism of $S_{g-1,2}$ (extend by the identity).  Given $f
\in \I(S_{g-1,1})$, we take any representative homeomorphism, think of
it as a homeomorphism of $S_{g-1,2}$, and call the resulting mapping
class $F$.  By construction, $F(d)=d$.  It is clear that $F$ maps to
$f$ under the map $\Mod(S_{g-1,2}) \to \Mod(S_{g-1,1})$ that forgets
one puncture of $S_{g-1,2}$.  It remains to check that $F$ is an
element of $\Stab_{\I(S_g)}(v)$.

We embed $S_{g-1,2}$ into $S_g$ so that the punctures correspond to
homotopic nonseparating curves (that is, ``uncut'').  If the isotopy
class corresponding to the punctures is $c$, we can extend $F$ to an
element of $\Mod(S_g)$ that is well-defined up to Dehn twists about
$c$.  Since the support of $F$ is on a surface of genus $g-1$ and one
puncture, and $F$ acts trivially on the homology of this subsurface,
we can choose a lift of $F$ to $\Mod(S_g)$ that lies in $\I(S_g)$.
This completes the proof.
\end{proof}


\subsection{Proof of infinite generation of top homology}

This section contains the proof of Theorem~\ref{main:top}, which states that
$H_{3g-5}(\I(S_g),\Z)$ is infinitely generated.

\begin{lemma}
\label{lem:top hom 1}
If $H_{3g-5}(\I(S_g))$ is infinitely generated, then 
$H_{3g-3}(\I(S_{g,1}))$ is infinitely generated.
\end{lemma}

\begin{proof}

We know that $\cd(\pi_1(S_g))=2$.  By
Theorem~\ref{main:torelli}, $\cd(\I(S_g)) = 3g-5$. 
Applying Fact~\ref{fact:spectral} to the
short exact sequence of Lemma~\ref{seq2}, and using the identification
$H_2(\pi_1(S_g)) = H_2(S_g) \cong \Z$, we find
\[ H_{3g-3}(\I(S_{g,1})) \cong H_{3g-5}(\I(S_{g}),\Z). \]
By Lemma~\ref{lem:trivial action}, the coefficients in the last homology group are
global.  Therefore, by Theorem~\ref{main:top},
$H_{3g-3}(\I(S_{g,1}))$ is infinitely generated.
\end{proof}

We record the following corollary of Lemma~\ref{lem:top hom 1} (use Lemma~\ref{seq2} and Fact~\ref{fact:cd ses}).

\begin{cor}
\label{cor:dim punc}
For $g \geq 2$, we have $\cd(\I(S_{g,1}))=3g-3$.
\end{cor}

\begin{lemma}
\label{lem:vertex inject}
If $v$ is any vertex of $\B(S_g)$, then $H_{3g-5}(\Stab_{\I(S_g)}(v))$
injects into $H_{3g-5}(\I(S_g))$.
\end{lemma}

\begin{proof}

We consider the Cartan--Leray spectral sequence associated to the
action of $\I(S_g)$ on $\B(S_g)$.  By Proposition~\ref{prop:dim
  lemma}, we know that $E^1_{p,q}=0$ whenever $p+q > 3g-5$.  By
Fact~\ref{fact:cartan}, the group $E^1_{0,3g-5} = \oplus H_{3g-5}(\Stab_{\I(S_g)}(v))$, where the
sum is over a set of representatives of the vertices of $\B(S_g)/\I(S_g)$, injects
into $H_{3g-5}(\I(S_g))$. The lemma follows.
\end{proof}

We are now ready to prove the theorem.

\begin{proof}[Proof of Theorem~\ref{main:top}]

We proceed by induction on the genus $g$.  Theorem~\ref{main:mess} tells us that $\I(S_2)$ is an infinitely generated free group, and
so $H_1(\I(S_2))$ is infinitely generated.  Let $g \geq 3$, and assume
the theorem is true for $g-1$, that is, $H_{3g-8}(\I(S_{g-1}))$ is
infinitely generated.

Let $v$ be a vertex of $\B(S_g)$ represented by a single curve in the class $x$.  By Lemma~\ref{lem:vertex inject}, it suffices to show that $H_{3g-5}(\Stab_{\I(S_g)}(v))$ is infinitely generated.

Consider the short exact sequence in Lemma~\ref{seq1}.  Since $K$ is
free, we have $\cd(K)=1$, and Corollary~\ref{cor:dim punc} gives
$\cd(\I(S_{g-1,1})) = 3g-6$.  Applying Fact~\ref{fact:spectral} to the
sequence in Lemma~\ref{seq1}, we obtain
\[ H_{3g-5}(\Stab_{\I(S_g)}(v)) \cong H_{3g-6}(\I(S_{g-1,1}),H_1(K)). \]
Thus, the theorem is reduced to showing that
$H_{3g-6}(\I(S_{g-1,1}),H_1(K))$ is infinitely generated.  By
Lemma~\ref{lem:global module}, there is a nontrivial submodule $M$ of $H_1(K)$ on which
$\I(S_{g-1,1})$ acts trivially.  Since $M$ is torsion free (it is a
subgroup of the first homology of a free group), the universal
coefficient theorem gives
\[ H_{3g-6}(\I(S_{g-1,1}),M) \cong H_{3g-6}(\I(S_{g-1,1})) \otimes M. \]
By the inductive hypothesis and Lemma~\ref{lem:top hom 1}, the latter is infinitely generated.

Now, the short exact sequence of modules
\[ 1 \to M \to H_1(K) \to H_1(K)/M \to 1 \]
induces a long exact sequence of homology groups, in which we find the following:
\[
\begin{array}{l}
 \cdots \to H_{3g-5}(\I(S_{g-1,1}),H_1(K)/M) \to
H_{3g-6}(\I(S_{g-1,1}),M) \\ \qquad \qquad \qquad \qquad \qquad \qquad    \to H_{3g-6}(\I(S_{g-1,1}),H_1(K)) \to
\cdots. \end{array} \]
By Corollary~\ref{cor:dim punc}, the first term shown is trivial.
Thus, the infinitely generated group $H_{3g-6}(\I(S_{g-1,1}),M)$ injects into
$H_{3g-6}(\I(S_{g-1,1}),H_1(K))$, and the theorem follows.
\end{proof}

In general, the vertices of $\B(S_g)$ correspond to basic cycles for $x$ supported on $k$ curves, with $1 \leq k \leq g$.  When $k < g$, the cohomological dimension of the vertex stabilizer is $3g-4-k$.  By an argument similar to that of Theorem~\ref{main:top}, the top dimensional integral homology of a vertex stabilizer ($k < g$) is infinitely generated.  For $k > 1$, there are infinitely many vertices of that type.  Therefore, to prove that $H_{3g-4-k}(\I(S_g),\Z)$ is infinitely generated, one would only have to show that the quotient of each vertex group by its incoming edge groups is nontrivial.

\begin{q}
Is is true that $H_{3g-4-k}(\I(S_g),\Z)$ is infinitely generated for $2 \leq k \leq g-1$?
\end{q}


\section{Torelli Mess subgroups}
\label{section:mess}

We now briefly describe Mess's proof of Theorem~\ref{thm:lower}.  We
do this for completeness, and also because Mess's argument is not
published in the exact form that gives Theorem~\ref{thm:lower}.  His
original paper is a preprint, available from Institut des Hautes
\'Etudes Scientifiques \cite{mess}.

The basic idea is to inductively define subgroups $\Gamma_g$ of
$\I(S_g)$ with the property that $\cd(\Gamma_g) = 3g-5$.  The theorem
then follows from Fact~\ref{fact:cd subgp} (monotonicity).

We employ the theory of Poincar\'e duality groups.  A group $\Gamma$
is a Poincar\'e duality group if the trivial $\Z\Gamma$-module $\Z$
admits a finite length projective resolution by finitely generated
$\Z\Gamma$-modules, and $H^i(\Gamma,\Z\Gamma)$ is trivial in all but one dimension, where it is group isomorphic to $\Z$; see \cite{jw} and \cite{rb} for background.  Fundamental
groups of closed aspherical manifolds are Poincar\'e duality groups,
for example.

We will need the following fact \cite[Theorem 3]{jw}.

\begin{theorem}
\label{thm:pd ses}
Given a short exact sequence of groups
\[ 1 \to K \to G \to Q \to 1 \]
where $K$ and $Q$ are Poincar\'e duality groups, it follows that $G$ is a Poincar\'e duality group and $\cd(G) = \cd(K) + \cd(Q)$.
\end{theorem}

To start the inductive process, we define $\Gamma_2$ to be the group
generated by a single Dehn twist about a separating curve in $S_2$.
Since $\Gamma_2 = \pi_1(S^1) \cong \Z$, we have that $\Gamma_2$ is a
Poincar\'e duality group.  Of course, $\cd(\Gamma_2)=1$.

We now assume that $\Gamma_{g}$ is given, and that it is a Poincar\'e duality group with cohomological dimension $3g-5$.  We will define $\Gamma_{g+1}$ in two steps.

Let $S_{g}'$ be the surface obtained from $S_{g}$ by removing an
open disk.  We consider the \emph{relative mapping class group}
$\Mod(S_{g}',\partial S_{g}')$, which is defined as
$\pi_0(\Homeo^+(S_{g}',\partial S_{g}'))$, where
$\Homeo^+(S_{g}',\partial S_{g}')$ is the subgroup of elements of
$\Homeo^+(S_{g}')$ that fix $\partial S_{g}'$ pointwise.  Let
$UT(S_{g})$ denote the unit tangent bundle for $S_{g}$.

We have the following generalization of the classical Birman exact
sequence due to Johnson \cite{dj1}.

\begin{theorem}[Relative Birman exact sequence]
\label{relative bes}
Let $g \geq 2$.  We have
\[ 1 \to \pi_1(UT(S_{g})) \to \Mod(S_{g}',\partial S_{g}') \to \Mod(S_{g}) \to 1. \]
\end{theorem}

It was noticed by Johnson that the entire image of $\pi_1(UT(S))$ lies
in the Torelli subgroup of $\Mod(S_{g}',\partial S_{g}')$ \cite{dj1}.
Therefore, the preimage of $\Gamma_g$ in
$\Mod(S_{g}',\partial S_{g}')$ is a subgroup of the Torelli group
$\I(S_{g}',\partial S_{g}')$.  We call this group $\Gamma_{g+1}'$.  By
Theorem~\ref{thm:pd ses}, $\Gamma_{g+1}'$ is a Poincar\'e duality
group of cohomological dimension $3(g+1)-5$ (since $UT(S_g)$ is a
closed aspherical 3-manifold, $\pi_1(UT(S_g))$ is a Poincar\'e duality group of
dimension 3).

The inclusion of $S_g'$ into $S_{g+1}$ induces an injective
homomorphism $\Mod(S_{g}',\partial S_{g}') \to \Mod(S_{g+1})$ (see,
e.g., \cite[Corollary 4.2]{pr}) which restricts to an inclusion on
the level of Torelli groups.  We define $\Gamma_{g+1}$ to be the image
of $\Gamma_{g+1}'$.  Since $\Gamma_{g+1}$ is isomorphic to
$\Gamma_{g+1}'$, we have that $\Gamma_{g+1}$ is a Poincar\'e duality
group of dimension $3(g+1)-5$, and this completes the proof of
Theorem~\ref{thm:lower}.

\p{Remark.} While the above argument certainly proves
Theorem~\ref{thm:lower}, we would like to point out that more is true.
Not only are the $\Gamma_g$ Poincar\'e duality groups, but they also
are fundamental groups of closed aspherical manifolds, with base
$\Gamma_{g-1}$ and fiber $UT(S_{g-1})$.  The proof of this stronger
statement is given by Ivanov \cite[Section 6.3]{nvi}.

\p{Remark.} Often the easiest way to get a lower bound on the
cohomological dimension of a group is to find a large abelian
subgroup (cf. Fact~\ref{fact:cd subgp}).  It is a theorem of Vautaw, however, that the largest
free abelian subgroup of $\I(S_g)$ has rank $2g-3$ \cite{wrv}.  In
Figure~\ref{maxpicx}, we exhibit such a subgroup in genus 5 (the
example generalizes to higher genus).  In that example, each generator is a Dehn twist
about a separating curve and so we see that we have $\Z^{2g-3}$ inside
$\K(S_g)$.  This gives the lower bound for $\cd(\K(S_g))$ given in the introduction.

\begin{figure}[htb]
\centerline{\includegraphics[scale=.25]{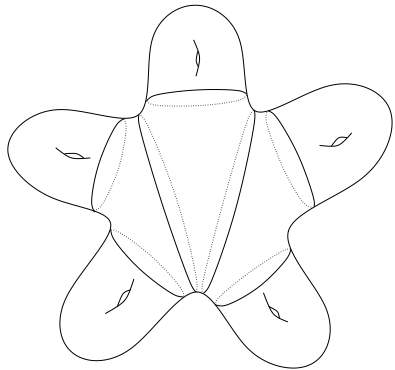}}
\caption{The Dehn twists about these curves generate a free abelian
  subgroup of maximal rank in $\I(S_5)$.}
\label{maxpicx}
\end{figure}

\p{Remark.} The construction of the Mess subgroups clearly illustrates
the discrepancy of $g$ between $\cd(\I(S_g))$ and $\vcd(\Mod(S_g))$.
To get a group of cohomological dimension $4g-5$ in $\Mod(S_g)$, Mess
simply augments the inductive construction of the $\Gamma_g$ by adding
the Dehn twist about a nonseparating curve in the new handle at each
stage.  This new Dehn twist generates a direct factor of the Mess
subgroup for $\Mod(S_g)$.  By only adding these Dehn twists on $g-k$
of the new handles, the construction gives the correct lower bound for Conjecture~\ref{conj:ik}.

\begin{figure}[htb]
\psfrag{x}{$a$}
\psfrag{d}{$d$}
\psfrag{y}{$b$}
\psfrag{a1}{$c_1$}
\psfrag{a2}{$c_2$}
\psfrag{a3}{$c_3$}
\psfrag{a4}{$c_4$}
\psfrag{a5}{$c_5$}
\centerline{\includegraphics[scale=.5]{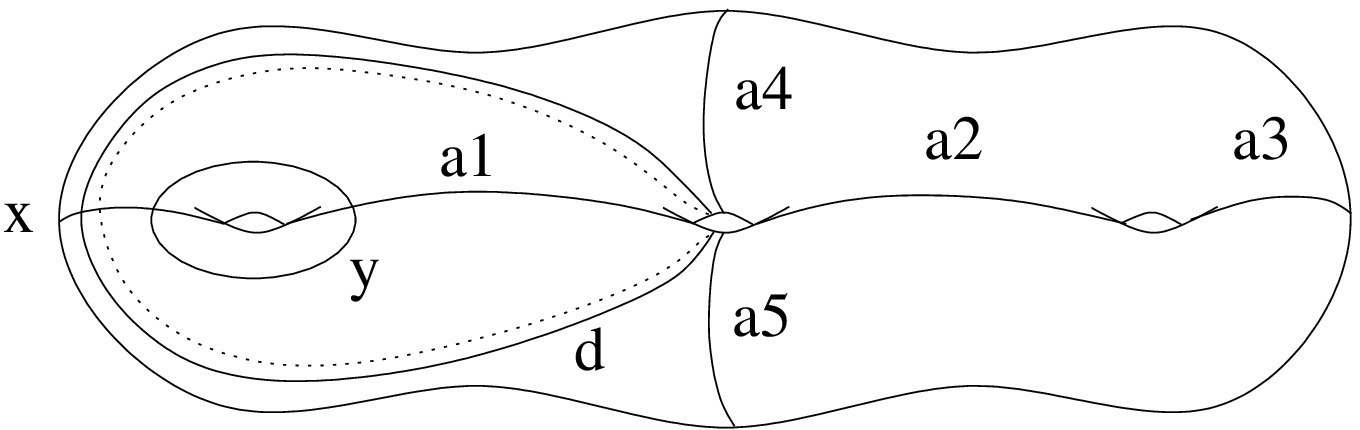}}
\caption{}
\label{counterconj2pic}
\end{figure}

On the other hand, we note that the conjecture cannot be proven by applying the
Quillen condition to the action of $\I^k(S_g)$ on $\B(S_g)$.  To see
this, we consider the group $\I^1(S_3)$, and the configuration shown in
Figure~\ref{counterconj2pic}; the handle being fixed homologically is
the one spanned by $a$ and $b$ (say $x=[a]$).  There is a 2-cell of $\B(S_3)$
corresponding to the $c_i$: $[c_1]+[c_2]+[c_3] = x$, and each of $c_4$ and
$c_5$ adds a ``borrowing dimension''.  The conjecture says that
$\cd(\I^k(S_g)) = 6$, so we would need that the stabilizer of
$\{c_i\}$ in $\I^1(S_3)$ has cohomological dimension at most 4.  On the contrary, we
can find a free abelian group of rank 5 in this stabilizer: the group
generated by $T_{c_2}$, $T_{c_3}$, $T_{c_4}$, $T_{c_5}$, and $T_d$.

\bibliographystyle{plain}
\bibliography{cdi}

\end{document}